\documentclass{article} 

\usepackage[a4paper, total={6.5in, 9in}]{geometry}

\usepackage{amsmath,amssymb,amsthm}
\usepackage{epsfig}
\usepackage[utf8]{inputenc}
\usepackage{psfrag,textcomp}
\usepackage[font=small,labelfont=bf,tableposition=top]{caption}
\usepackage[linesnumbered,lined,boxed,commentsnumbered]{algorithm2e}
\SetAlCapSkip{1ex}
\usepackage{color}
\usepackage{mathtools}
\usepackage{graphicx,import}
\usepackage{pgfplots}
\usepackage{caption} 
\usepackage{subcaption}
\usepackage{bm}
\usepackage{empheq}
\usepackage{slashbox}
\usepackage{array}
\usepackage{float}

\graphicspath{{}}

\DeclareCaptionLabelFormat{andtable}{#1~#2  \&  \tablename~\thetable}
\DeclareMathAlphabet\boldsymbolcal{OMS}{cmsy}{b}{n}

\newlength\figureheight
\newlength\figurewidth

\newcommand{\blueline}{\raisebox{2pt}{\tikz{\draw[-,blue,solid,line width = 1pt](0,0) -- (5mm,0);}}}
\newcommand{\redline}{\raisebox{2pt}{\tikz{\draw[-,red,solid,line width = 1pt](0,0) -- (5mm,0);}}}

\newtheorem{theorem}{Theorem}
\newtheorem{proposition}[theorem]{Proposition}
\newtheorem{remark}[theorem]{Remark}
\newtheorem{coro}[theorem]{Corollary}

\title{An error indicator-based adaptive reduced order model for nonlinear structural mechanics -- application to high-pressure turbine blades}

\author{Fabien Casenave $^{1}$, Nissrine Akkari $^{1}$\\
$^{1}$ Safran Tech, Modelling \& Simulation, Rue des Jeunes Bois,\\
Châteaufort, 78114 Magny-Les-Hameaux, France}

\begin{document}

\maketitle

\section*{Abstract}
The industrial application motivating this work is the fatigue computation of aircraft engines' high-pressure turbine blades. The material model involves nonlinear elastoviscoplastic behavior laws, for which the parameters depend on the temperature. For this application, the temperature loading is not accurately known and can reach values relatively close to the creep temperature: important nonlinear effects occur and the solution strongly depends on the used thermal loading. We consider a nonlinear reduced order model able to compute, in the exploitation phase, the behavior of the blade for a new temperature field loading. The sensitivity of the solution to the temperature makes {the classical unenriched proper orthogonal decomposition method} fail. In this work, we propose a new error indicator, quantifying the error made by the reduced order model in computational complexity independent of the size of the high-fidelity reference model. In our framework, when the {error indicator} becomes larger than a given tolerance, the reduced order model is updated using one time step solution of the high-fidelity reference model. The approach is illustrated on a series of academic test cases and applied on a setting of industrial complexity involving 5 million degrees of freedom, where the whole procedure is computed in parallel with distributed memory.

\section*{Keywords}
Nonlinear Reduced Order Model; Elastoviscoplastic behavior; Nonlinear structural mechanics; Proper Orthogonal Decomposition; Empirical Cubature Method; Error Indicator

\section{Introduction}
\label{sec:intro}

The application of interest for this work is the lifetime computation of aircraft engines' high-pressure turbine blades. Being located immediately downstream the combustion chamber, such parts undergo extreme thermal loading, with incoming fluid temperature higher than the material's melting temperature.  These blades are responsible for a large part of the maintenance budget of the engine, with temperature creep rupture and high-cycle fatigue~\cite{mazur2005474, cowles1996} as possible failure causes.
Various technological efforts have been spent to increase the durability of these blades as much as possible, such as  thermal barrier coatings~\cite{coatings}, advanced superalloys~\cite{superalloys} and complex internal cooling channels~\cite{coolingchannels1, coolingchannels2}, see Figure~\ref{fig:gaturbine} for a representation of a high-pressure turbine blade.
\begin{figure}[H]
  \centering
  \includegraphics[width=0.25\textwidth]{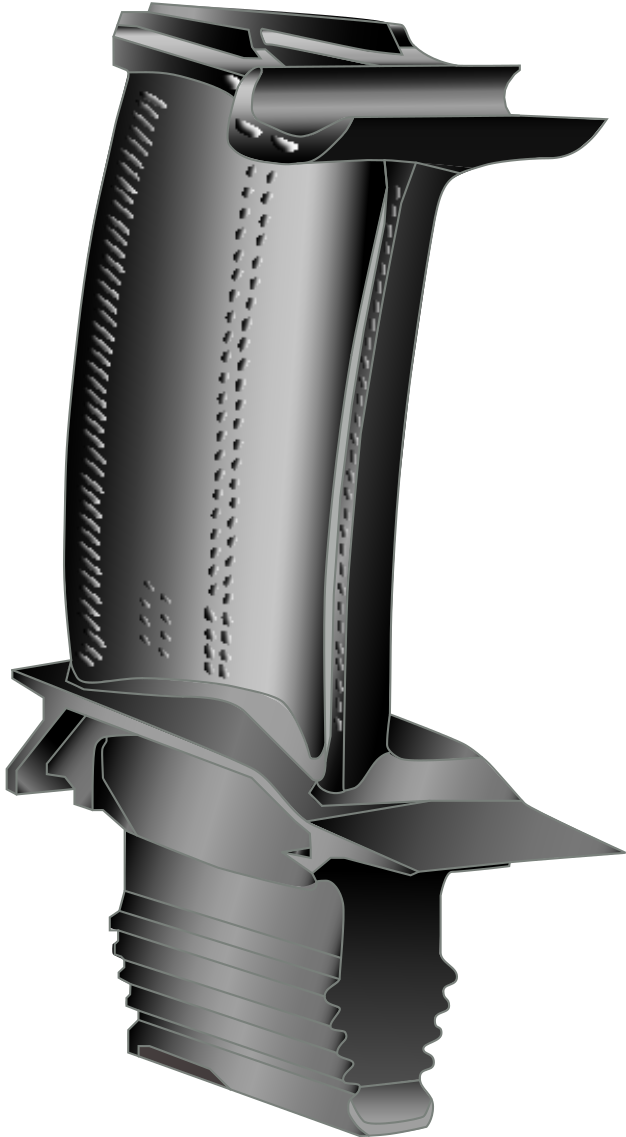}
  \caption{Illustration of a high-pressure turbine blade~\cite{gaturbine}. The internal channels create a protective layer of cool air to protect the outer surface of the blade.}
  \label{fig:gaturbine}
\end{figure}

Computing lifetime predictions for high-pressure turbine blades is a challenging task: meshes involve large numbers of degrees of freedom to account for local structures such as the internal cooling channels, the behavior laws are strongly nonlinear with many internal variables, and a large number of cycles has to be computed. Besides, the temperature loading is poorly known in the outlet section of the combustion chamber.
Our team has proposed in~\cite{ijnme} a nonintrusive reduced order model (ROM) strategy in parallel computation with distributed memory to mitigate the runtime issues: a domain decomposition method is used to compute the first cycle, and the reduced order model is used to speed up the computation of the following cycles, {which can be considered as a reduced order model-based temporal extrapolation}. As pointed out in~\cite{ijnme}, errors are accumulated during {this} temporal extrapolation. Moreover, quantifying the uncertainty on the lifetime with respect to some statistical description of the temperature loading using an already constructed reduced order model would introduce additional errors. In this context, error indicator-based enrichment of reduced order models is the topic of the present work.

Error estimation for reduced model predictions is a topic that receives interest in the scientific literature. The reduced basis method~\cite{maday2002priori, machiels2000output} for parametrized problems is a reduced order modeling method that intrinsically relies on efficient \textit{a posteriori} error bounds of the error between the reduced prediction and the reference high-fidelity (HF) solution. This method consists in a greedy enrichment of a current reduced order basis by the high-fidelity solution at the parametric value that maximizes the error bound on a rich sampling of the parametric space. Being intensively evaluated, the error bound must be computed in computational complexity independent of the number of degrees of freedom of the high-fidelity reference. Initially proposed for elliptic coercive partial differential equations~\cite{maday2002global}, where the error bound is the dual norm of the residual divided by a lower bound of the stability constant, the method has been adapted to problems of increased difficulty, with the derivation of certified error bounds for the Boussinesq equation~\cite{yano2014space}, the Burger's equation~\cite{OHLBERGER2013901}, the Navier-Stokes equations~\cite{manzoni2014efficient}. Numerical stability of such error estimations with respect to round-off error can be an issue in nonlinear problems, which was investigated in~\cite{casenave2012,casenave2014,buhr,chen2018robust}.

Even if it is not a requirement for their execution, error estimation is a desired feature for all the other reduced order modeling methods. In Proper Generalized Decomposition (PGD) methods~\cite{chinesta2013pgd}, error estimation based on the constitutive relation error method is available~\cite{ladeveze1999application, ladeveze2013toward, chamoin2017posteriori}. In Proper Orthogonal Decomposition (POD)-based reduced order modeling methods~\cite{POD1,POD2}, error estimators have been developed for linear-quadratic optimal control problems~\cite{troltzsch2009pod}, the approximation of mixte finite element problems~\cite{LUO20074184}, the optimal control of nonlinear parabolic partial differential equations~\cite{Kammann2012metho-18466}, and for the reduction of magnetostatic problems~\cite{henneron:hal-01107810} and Navier-Stokes equations~\cite{doi:10.1002/num.20393}. To reduce nonlinear problems, the POD has been coupled with reduced integration strategies called hyperreduction, for which error estimates in constitutive relation have been proposed~\cite{ryckelynck:hal-01237733,ryckelynck2013estimation}.
\textit{A priori} sensitivity studies for POD approximations of quasi-nonlinear parabolic equations are also available~\cite{Akkari2014}.

The contribution of this work consists in the construction of a new error indicator, adapted to the model order reduction of nonlinear structural mechanics, where we are interested in the prediction of the dual quantities such as the cumulated plasticity or the stress tensor. These dual quantities need a reconstruction step to be represented on the complete structure of interest, usually done using a Gappy-POD algorithm based on the reduced solution. We illustrate that the ROM-Gappy-POD residual of the quantities of interest is highly correlated to the error in our cases. From this observation, we propose a calibration step, based on the data computed during the \textit{offline} stage of the reduced order modeling, to construct an error indicator adapted to the considered problem and configuration. This error indicator is then used in enrichment strategies that improve the accuracy of the reduced order model prediction, when nonparametrized variations of the temperature field are considered in the \textit{online} stage.

The problem of interest, the evolution of an elastoviscoplastic body under a time-dependent loading, in presented in Section~\ref{sec:hfm}. Then, the \textit{\textit{a posteriori}} reduced order modeling of this problem is detailed in Section~\ref{sec:rom}. 
Section~\ref{sec:error_ind} presents the proposed error indicator, and the enrichment strategy based upon it. The performances of this error indicator and its ability to improve the quality of the reduced order model prediction via enrichment are illustrated in two numerical experiments involving elastoviscoplastic materials in Section~\ref{sec:num}.
Finally, conclusions and prospects are given in Section~\ref{sec:conclusion}.

\section{High-fidelity elastoviscoplastic model}
\label{sec:hfm}

We consider the model introduced in~\cite{ijnme}, which we briefly recall below for the sake of completeness. The structure of interest is noted $\Omega$ and its boundary $\partial\Omega$, where $\partial\Omega=\partial\Omega_D\cup\partial\Omega_N$ such that $\partial\Omega_D\cap\partial\Omega_N=\emptyset$, see Figure~\ref{fig:structure}.
\begin{center}
\def\svgwidth{0.4\textwidth}
\begin{figure}[H]
  \centering
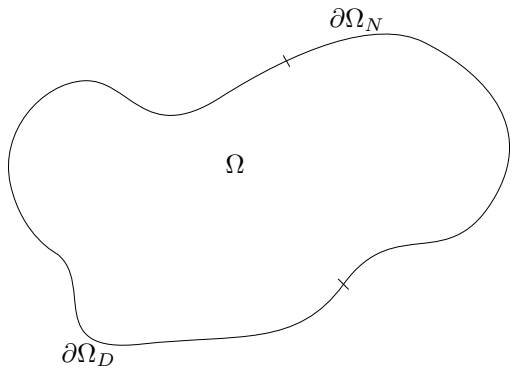
  \caption{Schematics of the considered structure $\Omega$.}
\label{fig:structure}
\end{figure}
\end{center}
Prescribed zero displacement are imposed on $\partial\Omega_D$, prescribed tractions $T_N$ are imposed on $\partial\Omega_N$ and volumic forces are imposed to the structure $\Omega$, in the form of a time-dependent loading. Assuming small deformations, the evolution of the structure $\Omega$ is governed by equations
\begin{subequations}
\begin{align}
&\epsilon(u)=\frac{1}{2}\left(\nabla u+ \nabla^T u\right)&\rm{in~}\textnormal{$\Omega$}\times[0,T]&\quad&\rm{(compatibility),}\label{eq:r1}\\
&{\rm div}\left(\sigma\right)+f=0&\rm{in~}\textnormal{$\Omega$}\times[0,T]&\quad&\rm{(equilibrium),}\label{eq:r2}\\
&\sigma=\sigma(\epsilon(u),y)&\rm{in~}\textnormal{$\Omega$}\times[0,T]&\quad&\rm{(behavior~law),}\label{eq:r3}\\
&u=0&\rm{in~}\partial\textnormal{$\Omega$}_D\times[0,T]&\quad&\rm{(prescribed~zero~displacement),}\label{eq:r4}\\
&\sigma\cdot n=T_N&\rm{in~}\partial\textnormal{$\Omega$}_N\times[0,T]&\quad&\rm{(prescribed~traction),}\label{eq:r5}\\
&u=0,y=0&\rm{in~}\textnormal{$\Omega$}~{\rm at~t=0}&\quad&\rm{(initial~condition),}\label{eq:r6}
\end{align}
\end{subequations}
where $\sigma$ is the Cauchy stress tensor, $\epsilon$ is the linear strain tensor, $n$ is the exterior normal on $\partial\Omega$, $y$ denotes the internal variables of the behavior law, and $u$ is the displacement solution.

Consider $H^1_0(\Omega)=\{v\in L^2(\Omega)|~\frac{\partial v}{\partial x_i}\in L^2(\Omega),~1\leq i\leq 3{\rm~and~}v|_{\partial\Omega_D}=0\}$. We introduce a finite element basis $\{\varphi_i\}_{1\leq i\leq N}$, such that $\mathcal{V}:={\rm Span}\left(\varphi_i\right)_{1\leq i\leq N}$ is a conforming approximation of $\left[H^1_0(\Omega)\right]^3$. {In what follows, bold symbols are used to refer to vectors.}
Using the Galerkin method, problem~\eqref{eq:r1}-\eqref{eq:r6} leads to a system of nonlinear equations, numerically solved using the following Newton algorithm:
\begin{equation}
\label{eq:Newton}
\displaystyle\frac{D\mathcal{F}}{Du}\left(u^k\right)\left(\bold{u}^{k+1}-\bold{u}^{k}\right)=-\boldsymbolcal{F}\left(u^k\right),
\end{equation}
where {$u^k\in\mathcal{V}$ is the k-th iteration of the discretized displacement field at the considered time-step and $\bold{u}^k=\left(u_i^k\right)_{1\leq i\leq N}\in\mathbb{R}^N$ is such that $\displaystyle u^k=\sum_{i=1}^Nu^k_i\varphi_i$},
\begin{equation}
\label{eq:TgtMat}
\displaystyle{\frac{D\mathcal{F}}{Du}\left({u}^k\right)}_{ij}=\int_{\Omega}\epsilon\left(\varphi_j\right):\mathcal{K}\left(\epsilon(u^k),y\right):\epsilon\left(\varphi_i\right), 1\leq i,j\leq N,
\end{equation}
where $\mathcal{K}\left(\epsilon(u^k),y\right)$ is the local tangent operator, and
\begin{equation}
\label{eq:problem}
\displaystyle {\mathcal{F}_i\left({u}^k\right)}=\int_\Omega\sigma\left(\epsilon({u}^k),y\right):\epsilon\left(\varphi_i\right)-\int_\Omega f\cdot\varphi_i-\int_{\partial\Omega_N}T_N\cdot\varphi_i, 1\leq i\leq N.
\end{equation}
The Newton algorithm stops when the norm of the residual divided by the norm of the external forces vector is smaller than a user-provided tolerance, denoted $\epsilon^{\rm HFM}_{\rm Newton}$.

In Equation~\eqref{eq:Newton}, $f$, $T_N$, $u^k$ and $y$ from~\eqref{eq:problem} are known quantities and contain the time-dependency of the solution. Notice that the computation of the functions $\displaystyle\left({u}^k,y\right)\mapsto \sigma\left(\epsilon({u}^k),y\right)$ and $\displaystyle\left({u}^k,y\right)\mapsto \mathcal{K}\left(\epsilon(u^k),y\right)$ requires solving ordinary differential equations, whose complexity depends on the behavior law modeling the considered material.

In our application, the quantities of interest are not the displacement fields $u$, but rather the dual quantities stress tensor field $\sigma$ and cumulated plasticity field, denoted $p$. The finite element software used to generate the high-fidelity solutions $u$ is Zebulon, which contains a Domain Decomposition solver able to solve large scale problems, and the {behavior} laws are computed using Z-mat; both solvers belong to the Z-set suite~\cite{zset}.

\section{Reduced Order Modeling}
\label{sec:rom}

Reduced order modeling techniques are usually decomposed in two stages: the \textit{offline} stage, where information from the high-fidelity model (HFM) is learned, and the \textit{online} stage, where the reduced order model is constructed and exploited. In the \textit{offline} stage occur computationally demanding tasks, whereas the \textit{online} stage is required to be efficient, in the sense that only operations in computational complexity independent of the number $N$ of degrees of freedom of the high-fidelity model are allowed.

In what follows, we consider \textit{\textit{a posteriori}} reduced order modeling, which means that our reduced model involves an efficient Galerkin method no longer written in the finite element basis $(\varphi_i)_{1\leq i\leq N}$, but on a reduced order basis  $(\psi_i)_{1\leq i\leq n}$, with $n\ll N$, adapted to the problem at hand.
To generate this basis, the high-fidelity problem~\eqref{eq:r1}-\eqref{eq:r6} is solved for given configurations. In the general case, the variations between the candidate configurations are quantified using a low-dimensional parametrization, leading to a parametrized reduced order model. 
In this work, we consider nonparametrized variations between the configurations of interest, which we call variability and denote $\mu$. The variability contains the time step, as well as a nonparametrized description of the configuration, which in our case is the loading referred as a label. For instance, $\mu=\{t=3, {\rm ``computation~1''}\}$, means that we consider the third time step of the configuration ``computation 1'', for which we have a description of the loading (center, axis and speed of rotation, temperature, and pressure fields in our applications).
{We denote $\mathcal{P}_{\rm off.}$ the set of variabilities encountered during the \textit{offline} stage.}

The reduced Newton algorithm reads
\begin{equation}
\label{eq:reducedNewton}
\displaystyle\frac{D\mathcal{F}_{\mu}}{Du}\left(\hat{u}^k_{\mu}\right)\left(\hat{\bold{u}}^{k+1}_{\mu}-\hat{\bold{u}}^{k}_{\mu}\right)=-\boldsymbolcal{F}_{\mu}\left(\hat{u}^k_{\mu}\right),
\end{equation}
where {$\hat{u}^k_{\mu}\in\hat{\mathcal{V}}:={\rm Span}\left(\psi_i\right)_{1\leq i\leq n}$ is the $k$-th iteration of the reduced displacement field for the considered time-step and $\hat{\bold{u}}^k_{\mu}=\left(\hat{u}^k_{\mu,i}\right)_{1\leq i\leq n}\in\mathbb{R}^n$ is such $\displaystyle\hat{u}^k_{\mu}=\sum_{i=1}^n\hat{u}^k_{\mu,i}\psi_i$},
\begin{equation}
\label{eq:redTgtMat}
\displaystyle{\frac{D\mathcal{F}_{\mu}}{Du}\left(\hat{u}^k_{\mu}\right)}_{ij}=\int_{\Omega}\epsilon\left(\psi_j\right):\mathcal{K}\left(\epsilon(\hat{u}^k_{\mu}),y_{\mu}\right):\epsilon\left(\psi_i\right), 1\leq i,j\leq n,
\end{equation}
and
\begin{equation}
\label{eq:redProblem}
\displaystyle {\mathcal{F}_{\mu,i}\left(\hat{u}^k\right)}=\int_\Omega\sigma\left(\epsilon(\hat{u}^k_{\mu}),y_{\mu}\right):\epsilon\left(\psi_i\right)-\int_\Omega f_{\mu}\cdot\psi_i-\int_{\partial\Omega_N}T_{N,\mu}\cdot\psi_i, 1\leq i\leq n.
\end{equation}
The reduced Newton algorithm stops when the norm of the reduced residual divided by the norm of the reduced external forces vector is smaller than a user-provided tolerance, denoted $\epsilon^{\rm ROM}_{\rm Newton}$. In~\eqref{eq:reducedNewton}-\eqref{eq:redProblem}, the \textit{online} variability $\mu$ consists in the considered time step, the pressure field $T_{N,\mu}$, the centrifugal effects $ f_{\mu}$, and the temperature field in the internal variables $y_{\mu}$.

Ensuring the efficiency of~\eqref{eq:reducedNewton} can be a complicated task, in particular for nonlinear problems, that requires methodologies recently proposed in the literature. For instance, the integrals in~\eqref{eq:redTgtMat} and~\eqref{eq:redProblem} are computed in computational complexity dependent on $N$ in the general case.
We briefly present the choices made in our previous work~\cite{ijnme}: the 
\textit{offline} stage is composed of the following steps
\begin{itemize}
\item data generation: this corresponds to the generation of the numerical approximation of the solutions to~\eqref{eq:r1}-\eqref{eq:r6}, using the Newton algorithm~\eqref{eq:Newton}. Multiple temporal solutions can be considered, for different loading conditions. The set of theses solutions $\{u_{\mu_i}\}_{1\leq i\leq N_c}$ is called the snapshots set.

\item data compression: this corresponds to the generation of the reduced order basis, usually obtained by looking for a hidden low-rank structure of the snapshots set. In this work, we consider the snapshot POD, see Algorithm~\ref{algo:snapshotPOD} and~\cite{POD1, POD2}.

\begin{algorithm}[H]
 \KwIn{tolerance $\epsilon_{\rm POD}$, snapshots set $\{u_{\mu_i}\}_{1\leq i\leq N_c}$}
 \KwOut{reduced order basis $\{\psi_i\}_{1\leq i\leq n}$}
 Compute the correlation matrix $\displaystyle C_{i,j}=\int_{\Omega}u_{\mu_i}\cdot u_{\mu_j}$, $1\leq i,j\leq N_c$
 
 Compute the $n$ largest eigenvalues $\lambda_i$, $1\leq i\leq n$, and associated orthonormal eigenvectors $\xi_i$, $1\leq i\leq n$, of $C$ such that  $n = \max\left(n_1,n_2\right)$, where $n_1$ and $n_2$ are respectively the smallest integers such that $\displaystyle\sum_{i=1}^{n_1} \lambda_i \geq \left(1-\epsilon_{\rm POD}^2\right)\sum_{i=1}^{N_c} \lambda_i$ and $\lambda_{n_2}\leq \epsilon_{\rm POD}^2 \lambda_{0}$

 Compute the reduced order basis $\displaystyle\psi_i(x)=\frac{1}{\sqrt{\lambda_i N_c}}\sum_{j=1}^{N_c}u_{\mu_j}(x)\xi_{i,j}$, $1\leq i\leq n$
 \caption{\label{algo:snapshotPOD}Data compression by snapshot POD.}
\end{algorithm}

\item operator compression: this step enables the efficient construction of~\eqref{eq:reducedNewton}, usually by replacing the computationally demanding integral evaluations by adapted approximation evaluated in computational complexity independent of $N$. In this work, we consider the Empirical Cubature Method (ECM, see~\cite{ECM}), a method close to the Energy Conserving Sampling and Weighting (ECSW, see~\cite{ECSW1, ECSW2, ECSW3}) proposed earlier. Consider the vector of reduced internal forces appearing in~\eqref{eq:redProblem}:
\begin{equation}
\label{eq:intForcesQuadrature}
\hat{F}^{\rm int}_{\mu,i}:=\int_\Omega\sigma\left(\epsilon(\hat{u}_{\mu}),y_{\mu}\right)(x):\epsilon\left(\psi_i\right)(x)dx\approx\sum_{e\in E}\sum_{k=1}^{n_e}\omega_k\sigma\left(\epsilon(\hat{u}_{\mu}),y_{\mu}\right)(x_k):\epsilon\left(\psi_i\right)(x_{k}), 1\leq i\leq n,
\end{equation}
where the right-hand side is the high-fidelity quadrature formula used for numerical evaluation. In~\eqref{eq:intForcesQuadrature}, the stress tensor $\sigma\left(\epsilon(\hat{u}_{\mu}),y_{\mu}\right)$ for the considered reduced solution $\hat{u}_{\mu}$ at variability $\mu$ and internal variables $y_{\mu}$ is seen as a function of space, and $E$ denotes the set of elements of the mesh, $n_e$ denotes the number of integration points for the element $e$, $\omega_k$ and $x_k$ are the integration weights and points of the considered element. 
The ECM consists in replacing this high-fidelity quadrature~\eqref{eq:intForcesQuadrature} by an approximation adapted to the snapshots $\{u_{\mu_i}\}_{1\leq i\leq N_c}$ and the reduced order basis $\{\psi_i\}_{1\leq i\leq n}$, and involving a small number of integration points:
\begin{equation}
\label{eq:intForcesQuadrature2}
\hat{F}^{\rm int}_{\mu,i}(t)\approx\sum_{k'=1}^{d}\hat{\omega}_{k'}\sigma\left(\epsilon(\hat{u}_{\mu}),y_{\mu}\right)(\hat{x}_{k'}):\epsilon\left(\psi_i\right)(\hat{x}_{k'}), 1\leq i\leq n,
\end{equation}
where $\displaystyle d\ll \sum_{e\in E}n_e$, the reduced integration points $\hat{x}_{k'}$, $1\leq k'\leq d$, are taken among the integration points of the high-fidelity quadrature~\eqref{eq:intForcesQuadrature} and the reduced integration weights $\hat{\omega}_{k'}$ are positive. 

We now briefly present how this reduced quadrature formula is obtained and we refer to~\cite{ijnme, ECM} for more details. We denote ${h}_q:=\sigma\left(\epsilon({u_{\mu_{(q//n)+1}}}),y\right):\epsilon\left(\psi_{(q\%n)+1}\right){\in L^2(\Omega)}$, where $//$  and $\%$ are respectively the quotient and the remainder of the Euclidean division, $\mathcal{Z}$ is a subset of $[1;N_G]$ of size $d$, {with $N_G$ the number of integration points},  and $J_{\mathcal{Z}}\in\mathbb{R}^{nN_c\times d}$ and ${\boldsymbol{g}}\in\mathbb{N}^{nN_c}$ are such that for all $1\leq q\leq nN_c$ and all $1\leq {k'}\leq d$, 
\begin{equation}
J_{\mathcal{Z}} = \Bigg({h}_q(x_{{\mathcal{Z}_{k'}}})\Bigg)_{1\leq q\leq nN_c,~{1\leq k'\leq d}},\qquad \boldsymbol{g} = \left(\int_{\Omega}{h}_q\right)_{1\leq q\leq nN_c},
\end{equation}
{where $\mathcal{Z}_{k'}$ denotes the $k'$-th element of  $\mathcal{Z}$ and where we recall that $n$ is the number of snapshot POD modes. Let $\hat{\boldsymbol{\omega}}\in{\mathbb{R}^{+}}^d$. From the introduced notation, $\displaystyle\left(J_{\mathcal{Z}}\hat{\boldsymbol{\omega}}\right)_q=\sum_{k'=1}^{d}\hat{\omega}_{k'}\sigma\left(\epsilon(u_{\mu_{(q//n)+1}}),y\right)(x_{\mathcal{Z}_{k'}}):\epsilon\left(\psi_{(q\%n)+1}\right)(x_{\mathcal{Z}_{k'}})$, $1\leq q\leq nN_c$, which is a candidate approximation for $\displaystyle \int_{\Omega}\sigma\left(\epsilon(u_{\mu_{(q//n)+1}}),y\right):\epsilon\left(\psi_{(q\%n)+1}\right) = g_q$, $1\leq q\leq nN_c$.}
The best reduced quadrature formula of length $d$ for the reduced internal forces vector is obtained as (c.f.~\cite[Equation (23)]{ECM})
\begin{equation}
\label{eq:optPb}
\left(\hat{\boldsymbol{\omega}},\mathcal{Z}\right)=\arg\underset{\hat{\boldsymbol{\omega}}'>0,\mathcal{Z}'\subset[1;N_G]}{\min}\left\|J_{\mathcal{Z}'}\hat{\boldsymbol{\omega}}'-\boldsymbol{g}\right\|_{{2}},
\end{equation}
where $\left\|\cdot\right\|_{{2}}$ stands for the Euclidean norm. Taking the length of the reduced quadrature formula in the objective function yields a NP-hard optimization problem, see~\cite[Section 5.3]{ECSW1}, citing~\cite{amaldi}. 
To produce a reduced quadrature formula in a controlled return time, we consider a Nonnegative Orthogonal Matching Pursuit algorithm, see~\cite[Algorithm 1]{fastNNOMP} and Algorithm~\ref{algo:NNOMP} below, a variant of the Matching Pursuit algorithm~\cite{matchingpursuit} tailored to the nonnegative requirement.

\begin{algorithm}[H]
 \KwIn{$J$, $b$, tolerance $\epsilon_{\rm Op. comp.}$}
 \KwOut{$\hat{\omega}_k$, $\hat{x}_k$, $1\leq k\leq d$}
 \textbf{Initialization:} $\mathcal{Z}=\emptyset$, $k'=0$, $\hat{\boldsymbol{\omega}}=0$ and $\boldsymbol{r_0}=\boldsymbol{g}$
 \While{$\left\|\boldsymbol{r_{k'}}\right\|_2>\epsilon\left\|\boldsymbol{g}\right\|_2$}{
  $\mathcal{Z}\leftarrow \mathcal{Z}\cup {\rm max~index} \left(J_{[1:N_G]}^T \boldsymbol{r_{k'}}\right)$
  
  $\hat{\boldsymbol{\omega}}\leftarrow \underset{\hat{\boldsymbol{\omega}}'>0}{\arg}\min\left\|\boldsymbol{g}-J_\mathcal{Z}\hat{\boldsymbol{\omega}}'\right\|_2^2$
  
  $\boldsymbol{r_{{k'}+1}}\leftarrow \boldsymbol{g}-J_\mathcal{Z} \hat{\boldsymbol{\omega}}$
  
 ${k'}\leftarrow {k'}+1$
  }
  $d\leftarrow {k'}$
  
  $\hat{x}_k:=x_{\mathcal{Z}_k}$, $1\leq k\leq d$
 \caption{\label{algo:NNOMP}Nonnegative Orthogonal Matching Pursuit.}
\end{algorithm}

A reduced quadrature is also used to accelerate the integral computation in~\eqref{eq:redTgtMat}. The remaining integral computations in~\eqref{eq:reducedNewton} are $\displaystyle\int_\Omega f_{\mu}\cdot\psi_i$ and $\displaystyle\int_{\partial\Omega_N}T_{N,\mu}\cdot\psi_i$. They do not depend on the current solution, but only on the loading of the \textit{online} variability $\mu$, which is no longer efficient for nonparametrized variabilities. However, in our context of large scale nonlinear mechanics, these integrals are computed very fast with respect to the ones requiring behavior law resolutions, see Remark~\ref{rq:efficiency}.
\end{itemize}

For the primal quantity displacement $u$, we can identify the solution of the reduced problem $\hat{\bold{u}}^k_{\mu}\in\mathbb{R}^n$ with the reconstruction on the complete domain $\Omega$: $\displaystyle\hat{u}^k_{\mu} = \sum_{i=1}^n\hat{u}^k_{\mu,i}\psi_i$. For the dual quantities, such identification does not exist. However, the behavior law has already been evaluated at the integration point of the reduced quadrature $\hat{x}_k$, $1\leq k\leq d$. Since the evaluations are computed during the resolution of the reduced problem, we denote them by hats: for instance for the cumulated plasticity, $\hat{\boldsymbol{p}}_{\mu}\in\mathbb{R}^{d}$ is such that $\hat{p}_{\mu,k}$ is computed by the \textit{online} evaluation of the behavior law solver at the reduced integration points $\hat{x}_k$, $1\leq k\leq d$. To recover the cumulated plasticity on the complete structure $\Omega$, a ROM-Gappy-POD procedure is used to reconstruct the fields on the complete domain, see Algorithms~\ref{algo:GappyPOD1}-\ref{algo:GappyPOD2} and~\cite{Everson95} for the original presentation of the Gappy-POD. In step 2 of Algorithm~\ref{algo:GappyPOD1}, EIM denotes the Empirical Interpolation method~\cite{EIM1,EIM2} and the set of integration point whose indices have been selected is still denoted $\{\hat{x}_k\}_{1\leq k\leq m^{p}}$, where $n^{p}\leq m^{p}\leq n^{p}+d$.
The dual quantities predicted by the reduced order model and reconstructed on the complete structure are denoted with tildes, for instance $\tilde{p}_{\mu}$ for the cumulated plasticity.

\begin{algorithm}[H]
 \KwIn{tolerance $\epsilon_{\rm Gappy-POD}$, cumulated plasticity snapshots set $\{p_{\mu_i}\}_{1\leq i\leq N_c}$, indices of the integration points of the reduced quadrature formula}
 \KwOut{indices for \textit{online} material law computation, ROM-Gappy-POD matrix}
Apply the snapshot POD (Algorithm~\ref{algo:snapshotPOD}) on the high-fidelity snapshots  $\{p_{\mu_i}\}_{1\leq i\leq N_c}$ to obtain the vectors $\psi^{p}_i$, $1\leq i\leq n^{p}$, orthonormal {with respect to the $L^2(\Omega)$-inner product}

{Apply the EIM to the collection of vectors $\psi^{p}_i$, $1\leq i\leq n^{p}$, to select $n^{p}$ distinct indices and complete (without repeat) this set of indices by the indices of the integration points of the reduced quadrature formula}

{Construct the matrix $M\in\mathbb{N}^{n^{p}\times n^{p}}$ such that $M_{i,j}=\sum_{k=1}^{m^{p}}\psi^{p}_i(\hat{x}_{k})\psi^{p}_j(\hat{x}_{k})$ (Gappy scalar product of the POD modes)}
 \caption{\label{algo:GappyPOD1}Dual quantity reconstruction of the cumulated plasticity~$p$: \textit{offline} stage of the ROM-Gappy-POD.}
\end{algorithm}

\begin{algorithm}[H]
 \KwIn{\textit{online} variability $\mu$, indices for \textit{online} material law computation, ROM-Gappy-POD matrix}
 \KwOut{reconstructed value for $p$ on the complete domain $\Omega$}
{Construct $\boldsymbol{b}_{\mu}\in\mathbb{R}^{n^{p}}$, where $b_{\mu,i}=\sum_{k=1}^{m^p}\psi^p_i(\hat{x}_{k})\hat{p}_{\mu,k}$, and $\hat{\boldsymbol{p}}_{\mu}\in\mathbb{R}^{m^p}$ is such that $\hat{p}_{\mu,k}$ is the \textit{online} prediction of $p$ at variability $\mu$ and integration point $\hat{x}_{k}$ (from the \textit{online} evaluation of the behavior law solver)}

{Solve the (small) linear system: $M\boldsymbol{z}_{\mu}=\boldsymbol{b}_{\mu}$}

{Compute the reconstructed value for $p$ on the complete subdomain $\Omega$ as $\tilde{p}_{\mu}:=\sum_{i=1}^{n^{p}} z_{\mu,i}\psi^{p}_i$}
 \caption{\label{algo:GappyPOD2}Dual quantity reconstruction of the cumulated plasticity~$p$: \textit{online} stage of the ROM-Gappy-POD.}
\end{algorithm}
The ROM-Gappy-POD reconstruction is well-posed, since the linear system considered in the \textit{online} stage of Algorithm~\ref{algo:GappyPOD2} is invertible, see~\cite[Proposition 1]{ijnme}.

An interesting feature of our framework is the ability to be used in sequential or in parallel with distributed memory. Independently of the high-fidelity solver, the solutions can be partitioned between some subdomains and the reduced order framework can treat the data in parallel. The MPI communications are limited to the computation of the scalar products in line 1 of Algorithm~\ref{algo:snapshotPOD} for the \textit{offline} stage, and the scalar products in~\eqref{eq:redTgtMat} and~\eqref{eq:redProblem} in the \textit{online} stage. Furthermore, these scalar products are well adapted to parallel processing: each process computes independently its contribution on its respective subdomain, and the interprocess communication is limited to an all-to-all transfer of a scalar. All the remaining operations in our framework are treated in parallel with no communication, in particular in the operator compression step, reduced quadrature formulae are constructed independently.
A natural use for the parallel framework is in coherence with Domain Decomposition solvers (potentially from commercial codes), which conveniently produce solutions partitioned in subdomains. Actually in our framework, the three steps of the \textit{offline} stage (data generation, data compression and operator compression), the \textit{online} stage, the post-treatment and the visualization are all treated in parallel with distributed memory, see~\cite{ijnme} for more details.

\section{A heuristic error indicator}
\label{sec:error_ind}

We look for an efficient error indicator in this context of general nonlinearities and nonparametrized variabilities.
In model order reduction techniques, error estimation is an important feature, that becomes interesting under the condition that it can be computed in complexity independent of the number of degrees of freedom $N$ of the high-fidelity model.

\subsection{{First results on errors and residuals}}

{We recall some notations introduced so far: bold symbols refer to vectors ($\boldsymbol{p}_\mu$ is the vector of components the value of the HF cumulated plasiticity field at reduced integration points), hats refer to quantities computed by the reduced order model ($\hat{u}_\mu$ is the reduced displacement and $\hat{\boldsymbol{p}}_\mu$ is the vector of components the value of the reduced cumulated plasticity at the reduced quadrature points), whereas tildes refer to dual quantities reconstructed by Gappy-POD (for instance $\tilde{p}$). Bold and tilde symbols, for instance $\tilde{\boldsymbol{p}}_\mu$, refer to the vectors of components the reconstructed dual quantities on the reduced integration points: $\tilde{p}_{\mu,k}=\tilde{p}_{\mu}(\hat{x}_{k})$, $1\leq k\leq m^p$. Notice that in the general case, $\tilde{\boldsymbol{p}}_\mu \neq \hat{\boldsymbol{p}}_\mu$: this discrepancy is at the base of our proposed error indicator. A table of notations is provided at the end of the document.}

A quantification for the prediction relative error is defined as 
\begin{equation}
\label{eq:error_1}
E^p_{\mu}:=
\left\lbrace
\begin{array}{ll}
\frac{\|p_{\mu}-\tilde{p}_{\mu}\|_{L^2(\Omega)}}{\|p_{\mu}\|_{L^2(\Omega)}}  & \mbox{if} \|p_{\mu}\|_{L^2(\Omega)}\neq 0\\
\frac{\|p_{\mu}-\tilde{p}_{\mu}\|_{L^2(\Omega)}}{{\underset{\mu\in\mathcal{P}_{\rm off.}}{\max}\|p_{\mu}\|_{L^2(\Omega)}}} & \mbox{otherwise,}\\
\end{array}\right.
\end{equation}
where we recall that $p_{\mu}$ and $\tilde{p}_{\mu}$ are respectively the high-fidelity and reduced predictions for the cumulated plasticity field at the variability $\mu$, {and $\mathcal{P}_{\rm off.}$ is the set of variabilities encountered during the \textit{offline} stage}.

Define the ROM-Gappy-POD residual as
\begin{equation}
\label{eq:error_ind_1}
\mathcal{E}^p_{\mu}:=
\left\lbrace
\begin{array}{ll}
\frac{\|\tilde{\boldsymbol{p}}_{\mu}-\hat{\boldsymbol{p}}_{\mu}\|_2}{\|\hat{\boldsymbol{p}}_{\mu}\|_2}  & \mbox{if} \|\hat{\boldsymbol{p}}_{\mu}\|_2\neq 0\\
\frac{\|\tilde{\boldsymbol{p}}_{\mu}-\hat{\boldsymbol{p}}_{\mu}\|_2}{{\underset{\mu\in\mathcal{P}_{\rm off.}}{\max}\|\hat{\boldsymbol{p}}_{\mu}\|_2}} & \mbox{otherwise,}\\
\end{array}\right.
\end{equation}
where $\|\cdot\|_2$ denotes the Euclidean norm. {Notice that the relative error $E^p_{\mu}$ involves fields and $L^2$-norms whereas the ROM-Gappy-POD residual $\mathcal{E}^p_{\mu}$ involves vectors of dual quantities in the set of reduced integration points and Euclidean norms}. In~\eqref{eq:error_ind_1}, $\|\tilde{\boldsymbol{p}}_{\mu}-\hat{\boldsymbol{p}}_{\mu}\|_2$ is the error between the \textit{online} evaluation of the cumulated plasticity by the behavior law solver: {$\hat{\boldsymbol{p}}_{\mu}$}, and the reconstructed prediction at the reduced integration points $\hat{x}_{k}$: {$\tilde{\boldsymbol{p}}_{\mu}$}, $1\leq k\leq m^p$.
Let $B\in\mathbb{R}^{m^p\times n^p}$ such that $B_{k,i}=\psi^{p}_i(\hat{x}_{k})$, $1\leq k\leq m^p$, $1\leq i\leq n^p$; by definition, {$\displaystyle\tilde{p}_{\mu,k}=\sum_{i=1}^{n^{p}} z_{\mu,i}\psi^{p}_i({\hat{x}_k})=\left(B\boldsymbol{z}_{\mu}\right)_k$, $1\leq k\leq m^p$}.  From Algorithm~\ref{algo:GappyPOD1}, $M=B^TB$ and from Algorithm~\ref{algo:GappyPOD2}, $\boldsymbol{b}_{\mu} = B^T\hat{\boldsymbol{p}}_{\mu}$, so that $\boldsymbol{z}_{\mu} = \left(B^TB\right)^{-1}B^T\hat{\boldsymbol{p}}_{\mu}$, which is the solution of the following unconstrained least-square optimization: $\boldsymbol{z}_{\mu}:=\underset{\boldsymbol{z}'\in\mathbb{R}^{n}}{\rm arg}{\rm min}\|B\boldsymbol{z}'-\hat{\boldsymbol{p}}_{\mu}\|_2^2$. Hence, {in~\eqref{eq:error_ind_1}}, $\|\tilde{\boldsymbol{p}}_{\mu}-\hat{\boldsymbol{p}}_{\mu}\|_2$ is the norm of the residual of the considered least-square optimization.

Suppose $K:=\{p_\mu,{\rm~for~all~possible~variabilities~}\mu\}$ is a compact subset of $L^2(\Omega)$ and define {the Kolmogorov $n$-width by $d_{n}(K)_{L^2(\Omega)}:=\underset{{\rm dim}(W)=n}{\inf}~d(K,W)_{L^2(\Omega)}$, where $d(K,W)_{L^2(\Omega)}:=\underset{v\in K}{\sup}~\underset{w\in W}{\inf}~\|v-w\|_{L^2(\Omega)}$, with $W$ a finite-dimensional subspace of $L^2(\Omega)$. The Kolmogorov $n$-width is an object from approximation theory; a presentation and discussion in a reduced order modeling context can be found in~\cite{madaykolmogorov}.} 
Denote also $\boldsymbol{\Pi}_\mu:=\left(\left(p_\mu,\psi^{p}_i\right)_{L^2(\Omega)}\right)_{1\leq i\leq n^p}\in\mathbb{R}^{n^p}$, where we recall that $\left\{\psi^{p}_i\right\}_{1\leq i\leq n^p}$ are the Gappy-POD modes obtained by Algorithm~\ref{algo:GappyPOD1} {and where $\left(\cdot,\cdot\right)_{L^2(\Omega)}$ denotes the $L^2(\Omega)$ inner-product}. 
{All the dual quantities being computed by the high-fidelity solver at the $N_G$ integration points, they have finite values at these points.
Unlike the primal displacement field, the dual quantity are not directly expressed in a finite element basis, but through their values on the integration points. For pratical manipulations, we express the dual quantity fields as a constant on each polyhedron obtained as a Voronoi diagram in each element of the mesh, with seeds the integration points; the constants corresponding to the value of the dual quantity on the corresponding integration point.
}

{We first control the numerator in the relative error $E^p_\mu$ with respect to the numerator in the ROM-Gappy-POD residual $\mathcal{E}^p_\mu$ in Proposition~\ref{estimation}.}

\begin{proposition}
\label{estimation}
There exist two positive constants $C_1$ and $C_2$ independent of $\mu$ (but dependent on $n^p$) such that
\begin{equation}
\label{eq:estimation}
\left\|p_{{\mu}}-\tilde{p}_\mu\right\|^2_{L^2(\Omega)}\leq  C_1\|B\boldsymbol{z}_\mu-\hat{\boldsymbol{p}}_{\mu}\|_2^2 + C_1 \|\boldsymbol{p}_{\mu}-\hat{\boldsymbol{p}}_{\mu}\|_2^2 + C_2 d(K, {\rm Span}\{\psi_i^p\}_{1\leq i\leq n^p})_{L^2(\Omega)}^2.
\end{equation}
\end{proposition}

\begin{proof}
There holds
\begin{subequations}
\begin{align}
\left\|p_{{\mu}}-\tilde{p}_\mu\right\|^2_{L^2(\Omega)}&\leq 2~\left\|\sum_{i=1}^{n^p}\left(\left(p_\mu,\psi^{p}_i\right)_{L^2(\Omega)}-z_{\mu,i}\right)\psi^{p}_i\right\|^2_{L^2(\Omega)}+2~\left\|p_\mu-\sum_{i=1}^{n^p}\left(p_\mu,\psi^{p}_i\right)_{L^2(\Omega)}\psi^{p}_i\right\|^2_{L^2(\Omega)}\label{eq:ineq1}\\
&= 2\sum_{i=1}^{n^p}\left(\left(p_\mu,\psi^{p}_i\right)_{L^2(\Omega)}-z_{\mu,i}\right)^2+2~\underset{w\in{\rm Span}\{\psi^{p}_i\}_{1\leq i\leq n^p}}{\inf}~\left\|p_\mu-w\right\|^2_{L^2(\Omega)}\label{eq:ineq2}\\
&\leq 2\sum_{i=1}^{n^p}\left(\Pi_{\mu,i}-z_{\mu,i}\right)^2+2~\underset{v\in K}{\sup}~\underset{w\in{\rm Span}\{\psi^{p}_i\}_{1\leq i\leq n^p}}{\inf}~\left\|v-w\right\|^2_{L^2(\Omega)}\label{eq:ineq3}\\
 &= 2\left\|M^{-1}M\left(\boldsymbol{\Pi}_{\mu}-\boldsymbol{z}_\mu\right)\right\|_2^2+2d(K, {\rm Span}\{\psi_i^p\}_{1\leq i\leq n^p})_{L^2(\Omega)}^2\label{eq:ineq4}\\
 &= 2\left\|M^{-1}B^T\left(B\boldsymbol{\Pi}_{\mu}-\boldsymbol{p}_{\mu}+\boldsymbol{p}_{\mu}-\hat{\boldsymbol{p}}_{\mu}+\hat{\boldsymbol{p}}_{\mu}-B\boldsymbol{z}_\mu\right)\right\|_2^2+2d(K, {\rm Span}\{\psi_i^p\}_{1\leq i\leq n^p})_{L^2(\Omega)}^2\label{eq:ineq5}\\
& \leq 6\left\|M^{-1}B^T\right\|^2_2\left(\|B\boldsymbol{\Pi}_{\mu}-\boldsymbol{p}_{\mu}\|_2^2+\|\boldsymbol{p}_{\mu}-\hat{\boldsymbol{p}}_{\mu}\|_2^2+\|B\boldsymbol{z}_\mu-\hat{\boldsymbol{p}}_{\mu}\|_2^2\right)+2d(K, {\rm Span}\{\psi_i^p\}_{1\leq i\leq n^p})_{L^2(\Omega)}^2\label{eq:ineq6}\\
& \leq C_1\|B\boldsymbol{z}_\mu-\hat{\boldsymbol{p}}_{\mu}\|_2^2 + C_1 \|\boldsymbol{p}_{\mu}-\hat{\boldsymbol{p}}_{\mu}\|_2^2 + C_2 d(K, {\rm Span}\{\psi_i^p\}_{1\leq i\leq n^p})_{L^2(\Omega)}^2\label{eq:ineq7},
\end{align}
\end{subequations}
{where the triangular inequality and the Jensen inequality on the square function have been applied in~\eqref{eq:ineq1}, and between~\eqref{eq:ineq5} and~\eqref{eq:ineq6}.
In~\eqref{eq:ineq7}, the} term $\|B\boldsymbol{\Pi}_{\mu}-\boldsymbol{p}_{\mu}\|_2^2$ has been incorporated in the term\\
$C_2 d(K, {\rm Span}\{\psi_i^p\}_{1\leq i\leq n^p})_{L^2(\Omega)}^2$. This can be done {since}
\begin{equation}
\begin{aligned}
\|B\boldsymbol{\Pi}_{\mu}-\boldsymbol{p}_{\mu}\|_2^2= &{\sum_{k=1}^{m^p}\left(p_\mu(\hat{x}_k)-\sum_{i=1}^{n^p}\left(p_\mu,\psi^{p}_i\right)_{L^2(\Omega)}\psi^{p}_i(\hat{x}_k)\right)^2}\\
\leq &{\frac{1}{\underset{1\leq k'\leq m^p}{\min}\nu_{k'}}\sum_{k=1}^{N_g}\nu_k\left(p_\mu(x_k)-\sum_{i=1}^{n^p}\left(p_\mu,\psi^{p}_i\right)_{L^2(\Omega)}\psi^{p}_i(x_k)\right)^2}\\
= & \frac{1}{\underset{1\leq k'\leq m^p}{\min}\nu_{k'}} \int_{\Omega}\left(p_\mu(x)-\sum_{i=1}^{n^p}\left(p_\mu,\psi^{p}_i\right)_{L^2(\Omega)}\psi^{p}_i(x)\right)^2dx\\
\leq & \frac{1}{\underset{1\leq k'\leq m^p}{\min}\nu_{k'}} d(K, {\rm Span}\{\psi_i^p\}_{1\leq i\leq n^p})_{L^2(\Omega)}^2,
\end{aligned}
\end{equation}
where $\nu_k$ denotes the volume of the cell of the Voronoi diagram associated with integration point $\hat{x}_k$.
\end{proof}

{We now control the numerator in the ROM-Gappy-POD residual $\mathcal{E}^p_\mu$ with respect to the numerator in the relative error $E^p_\mu$  in Proposition~\ref{estimation}, leading to Corollary~\ref{cor}, which provides a sense a consistency: without any error in the reduced prediction, the ROM-Gappy-POD residual $\mathcal{E}^p_\mu$ is zero.}
\begin{proposition}
\label{estimation2}
There exist two positive constants $K_1$ and $K_2$ independent of $\mu$ such that
\begin{equation}
\label{eq:estimation2}
\|{\tilde{\boldsymbol{p}}_{\mu}-\hat{\boldsymbol{p}}_{\mu}}\|_2^2\leq K_1 \left\|p_{{\mu}}-\tilde{p}_\mu\right\|^2_{L^2(\Omega)} + K_2 \|\boldsymbol{p}_{\mu}-\hat{\boldsymbol{p}}_{\mu}\|_2^2 .
\end{equation}
\end{proposition}

\begin{proof}
There holds
\begin{equation}
\begin{aligned}
\|{\tilde{\boldsymbol{p}}_{\mu}-\hat{\boldsymbol{p}}_{\mu}}\|_2^2&\leq 2~\|B\boldsymbol{z}_\mu-\boldsymbol{p}_{\mu}\|_2^2+2~\|\boldsymbol{p}_{\mu}-\hat{\boldsymbol{p}}_{\mu}\|_2^2\\
&\leq \frac{2}{\underset{1\leq k'\leq m^p}{\min}\nu_{k'}}\sum_{k=1}^{m^p}\nu_k\left(p_\mu(\hat{x}_k)-\sum_{i=1}^{n^p}\boldsymbol{z}_{\mu,i}\psi^{p}_i(\hat{x}_k)\right)^2+2~\|\boldsymbol{p}_{\mu}-\hat{\boldsymbol{p}}_{\mu}\|_2^2\\
&\leq \frac{2}{\underset{1\leq k'\leq m^p}{\min}\nu_{k'}}\int_{\Omega}\left(p_\mu(x)-\sum_{i=1}^{n^p}\boldsymbol{z}_{\mu,i}\psi^{p}_i(x)\right)^2dx+2~\|\boldsymbol{p}_{\mu}-\hat{\boldsymbol{p}}_{\mu}\|_2^2\\
&{=} K_1 \left\|p_{{\mu}}-\tilde{p}_\mu\right\|^2_{L^2(\Omega)} + K_2 \|\boldsymbol{p}_{\mu}-\hat{\boldsymbol{p}}_{\mu}\|_2^2.
\end{aligned}
\end{equation}
\end{proof}

\begin{coro}
\label{cor}
Suppose that the reduced solution is exact up to the considered time step at the online variability $\mu$: $p_{{\mu}}=\tilde{p}_\mu$ in $L^2(\Omega)$. In particular, the behavior law solver has been evaluated with the exact strain tensor and state variables at the integration points  $x_{k}$, leading to $\hat{p}_{\mu}(\hat{x}_{k})=p_{\mu}(\hat{x}_{k})$, $1\leq k\leq m^d$. From Proposition~\ref{estimation2}, $\|{\tilde{\boldsymbol{p}}_{\mu}}-\hat{\boldsymbol{p}}_{\mu}\|_2=0$, and $\mathcal{E}_{\mu}^p=0$.
\end{coro}

\subsection{{A calibrated error indicator}}

As we will illustrate in Section~\ref{sec:num}, the evaluations of the  ROM-Gappy-POD residual $\mathcal{E}^p_{\mu}$~\eqref{eq:error_ind_1} and the error $E^p_{\mu}$~\eqref{eq:error_1} are very correlated in our numerical simulations. Our idea is to exploit this correlation by training a Gaussian process regressor for the function $\mathcal{E}^p_{\mu}\mapsto {E}^p_{\mu}$.
At the end of the \textit{offline} stage, we propose to compute reduced predictions at variability values $\{\mu_i\}_{1\leq i\leq N_c}$ encountered during the data generation step, and the corresponding couples $\left(E^p_{\mu_i}, \mathcal{E}^p_{\mu_i}\right)$, $1\leq i\leq N_c$. A Gaussian process regressor is trained on these values and we define an approximation function
\begin{equation}
\label{eq:indicator}
\mathcal{E}^p_{\mu}\mapsto {\rm Gpr}^p(\mathcal{E}^p_{\mu})
\end{equation}
 for the error ${E}^p_{\mu}$ at variability $\mu$ as the mean plus 3 times the standard deviation of the predictive distribution at the query point $\mathcal{E}^p_{\mu}$: this is our proposed error indicator. If the dispersion around the learning data is small for certain values $\mathcal{E}^p_{\mu}$, then adding 3 times the standard deviation will not change very much the prediction, whereas for values with large dispersions of the learning data, this correction aims to provide an error indicator larger than the error. We use the GaussianProcessRegressor python class from scikit-learn~\cite{scikit-learn}.
Notice that although some operations in computational complexity dependent on $N$ are carried-out, we are still in the \textit{offline} stage, and they are much faster than the resolutions of the large size systems of nonlinear equations~\eqref{eq:Newton}.
If the \textit{offline} stage is correctly carried-out and since $\mathcal{E}^p_{\mu}$ is highly correlated with the error, only small values for $\mathcal{E}^p_{\mu}$ are expected to be computed. Hence, in order to train the Gaussian process regressor correctly for larger values of the error, the reduced Newton algorithm~\eqref{eq:reducedNewton} is solved with a large tolerance $\epsilon^{\rm ROM}_{\rm Newton}=0.1$. We call these operations ``calibration of the error indication'', see~Algorithm~\ref{algo:calibration} for a description and Figure~\ref{fig:workflow_offline} for a presentation of the workflow featuring this calibration step.

\begin{algorithm}[H]
 \KwIn{outputs of the data generation, data compression and operator compression steps of Section~\ref{sec:rom}}
 \KwOut{Approximation function $\mathcal{E}^p_{\mu}\mapsto{\rm Gpr}^p(\mathcal{E}^p_{\mu})$ of the error ${E}^p_{\mu}$}
 \textbf{Initialization:} $\mathcal{X}=\emptyset$

\For{$i\gets 1$ \KwTo $N_c$}{
  Construct and solve the reduced problem~\eqref{eq:reducedNewton} with $\epsilon^{\rm ROM}_{\rm Newton}=0.1$

  Compute the reconstructed plasticity $\tilde{p}_{\mu_i}$ using Algorithm~\ref{algo:GappyPOD2} and  $\mathcal{E}^p_{\mu_i}$ using~\eqref{eq:error_ind_1}
    
  Compute the error $E^p_{\mu_i}$ using~\eqref{eq:error_1}

  $\mathcal{X}\leftarrow \mathcal{X}\cup \left(\mathcal{E}^p_{\mu_i}, E^p_{\mu_i}\right)$
}

Construct an approximation function $\mathcal{E}^p_{\mu}\mapsto{\rm Gpr}^p(\mathcal{E}^p_{\mu})$ of the error ${E}^p_{\mu}$ using a Gaussian process regression and the data from $\mathcal{X}$
 \caption{\label{algo:calibration}Calibration of the error indicator.}
\end{algorithm}

\begin{figure}[H]
  \centering
  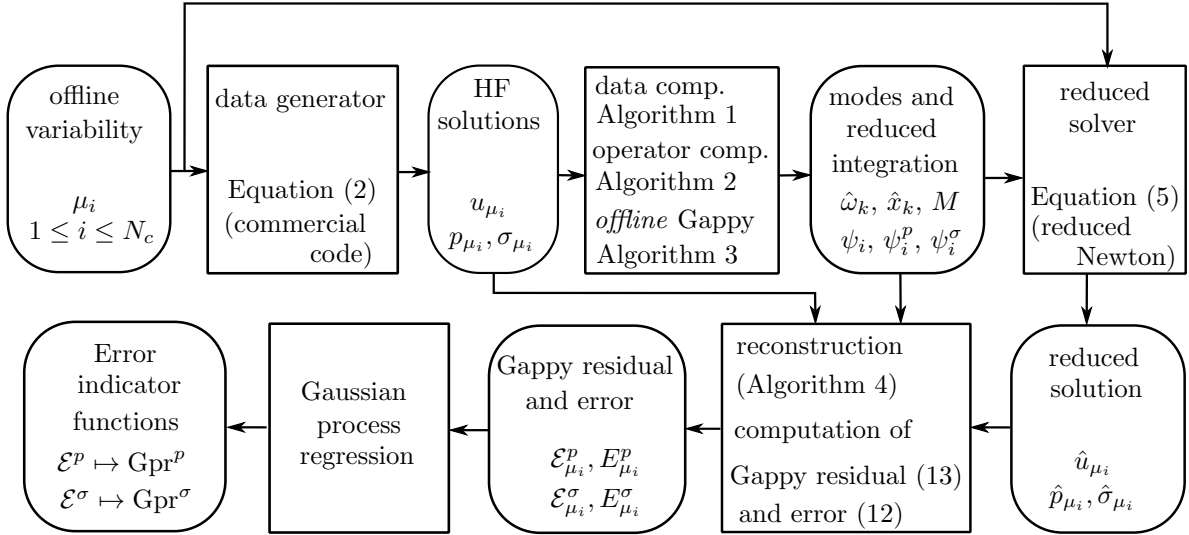
  \caption{Workflow for the \textit{offline} stage with error indicator calibration.}
  \label{fig:workflow_offline}
\end{figure}

We recall that in model order reduction, the original hypothesis is the existence of a low-dimensional vector space where an acceptable approximation of the high-fidelity solution lies. The hypothesis is formalized under a rate of decrease for the Kolmogorov $n$-width with respect to the dimension of this vector space. The same hypothesis is made when using the Gappy-POD to reconstruct the dual quantities, which are expressed as a linear combination of constructed modes. For both the primal and dual quantities, the modes are computed by searching some low-rank structure of the high-fidelity data.
The coefficients of the linear combination for reconstructing the primal quantities are given by the solution of the reduced Newton algorithm~\eqref{eq:reducedNewton}. After convergence, the residual is small, even in cases where the reduced order model exhibits large errors with respect to the high-fidelity reference: this residual gives no information on the distance between the reduced solution and the high-fidelty finite element space. However, in the \textit{online} phase of the ROM-Gappy-POD reconstruction in Algorithm~\ref{algo:GappyPOD2}, the coefficients $\hat{p}_{\mu,k}$ contain information from the high-fidelity behavior law solver. Moreover, an overdetermined least-square is solved, which can provide a nonzero residual that implicitly contains this information from the high-fidelity behavior law solver: namely the distance between the prediction from the behavior law and the vector space spanned by the Gappy-POD modes (restricted to the reduced integration points): this is the term $\|B\boldsymbol{z}_\mu-\hat{\boldsymbol{p}}_{\mu}\|_2$ in~\eqref{eq:estimation}. Hence, the ability of the \textit{online} variability to be expressed on the Gappy-POD modes is monitored through the behavior law solver on the reduced integration points. When the ROM is solved for an \textit{online} variability not included in the \textit{offline} variabilities, then the new physical solution cannot be correctly interpolated using the POD and Gappy-POD modes: hence, the ROM-Gappy-residual becomes large. 
From Proposition~\ref{estimation2}, if $\|B\boldsymbol{z}_\mu-\hat{\boldsymbol{p}}_{\mu}\|_2=\|\tilde{\boldsymbol{p}}_{\mu}-\hat{\boldsymbol{p}}_{\mu}\|_2$ is large, then the global error $\left\|p_{{\mu}}-\tilde{p}_\mu\right\|_{L^2(\Omega)}$ and/or the error at the reduced integration points $\hat{x}_k$ is large, which makes $\|B\boldsymbol{z}_\mu-\hat{\boldsymbol{p}}_{\mu}\|_2$ a good candidate for a enrichement criterion for the ROM. 
A limitation of the error indicator can occur if the \textit{online} variability activates strong nonlinearities on areas containing no point from the reduced integration scheme, namely through the term $C_2 d(K, {\rm Span}\{\psi_i^p\}_{1\leq i\leq n^p})_{L^2(\Omega)}^2$  in~\eqref{eq:estimation}.

{We recall that the error indicator~\eqref{eq:indicator} is a regression of the function $\mathcal{E}^p_{\mu}\mapsto {E}^p_{\mu}$. In the \textit{online} phase, we only need to evaluate $\mathcal{E}^p_{\mu}$ and do not require any estimation for the other terms and constants appearing in Propositions~\ref{estimation} and~\ref{estimation2}.}

Equipped with an efficient error indicator, we are now able to assess the quality of the approximation made by the reduced order model in the \textit{online} phase. If the error indicator is too large, {an enrichment step occurs:} the high-fidelity model is used to compute a new high-fidelity snapshot, which is used to update the POD and Gappy-POD basis, as well as the reduced integration schemes. Notice that {for the enrichment steps to be computed}, the displacement field and all the state variables of the previous time step need to be reconstructed on the complete mesh $\Omega$ to provide the high-fidelity solver with the correct material state. The workflow for the \textit{online} stage with enrichment is presented in Figure~\ref{fig:workflow_online}.

\begin{figure}[H]
  \centering
  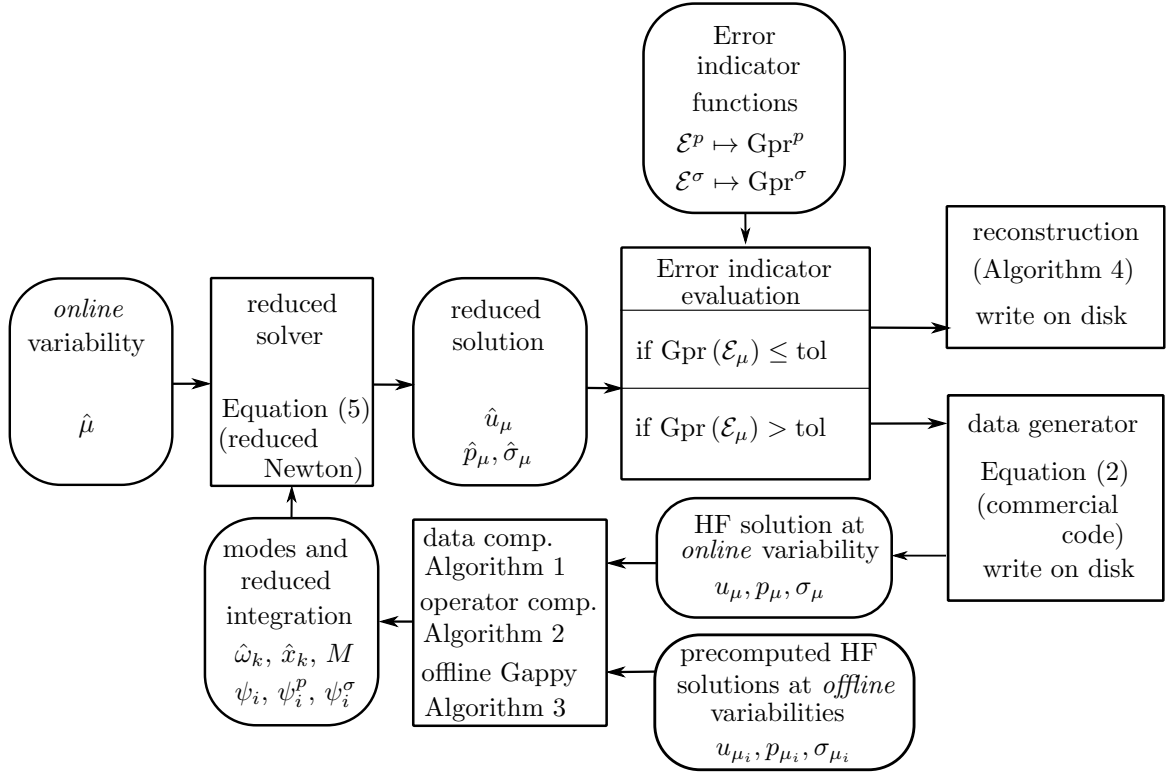
  \caption{Workflow for the \textit{online} stage with enrichment.}
  \label{fig:workflow_online}
\end{figure}

\begin{remark}[\textit{online} efficiency]
\label{rq:efficiency}
The computation of the ROM-Gappy-POD residual~\eqref{eq:error_ind_1} is efficient, since $\tilde{\boldsymbol{p}}_{\mu}$ and $\hat{\boldsymbol{p}}_{\mu}$ are already computed for the reconstruction, and {$m^p$ depending only on the approximation of $\sigma:\epsilon$ and $p$, it} is independent of $N$. The evaluation of ${\rm Gpr}^p(\mathcal{E}^p_{\mu})$ is also in computational complexity independent of $N$.

If the enrichment is activated during the \textit{online} phase, a high-fidelity solution is computed, {which is a computationally demanding task}. This is the price to add high-fidelity information {in} the exploitation phase. We will see in Section~\ref{sec:num} that without this enrichment in our applications, the considered \textit{online} variability on the temperature field strongly degrades the accuracy of the reduced order model prediction. The nonparametrized variability also induces online pretreatments {in computational complexity depending on $N$}, namely the precomputation of $\displaystyle\int_\Omega f_{\mu}\cdot\psi_i$ and $\displaystyle\int_{\partial\Omega_N}T_{N,\mu}\cdot\psi_i$ in~\eqref{eq:redProblem}, which is in practice much faster than other integrals that require behavior law resolutions.

Notice that the \textit{online} stage can be further optimized by replacing the data compression and \textit{offline} Gappy-POD steps by incremental variants, such as the incremental POD~\cite{incrementalPOD}. For the operator compression, the Nonnegative Orthogonal Matching Pursuit described in Algorithm~\ref{algo:NNOMP} is not restarted from zero, but initialized by the current reduced quadrature scheme. Notice also that for the moment, the reduced order model is enriched using a complete {precomputed} reference high-fidelity computation, {so that no speedup is obtained in practice}. We still need to consider restart strategies to call the high-fidelity solver only at the time step of enrichment, from a complete mechanical state reconstructed from the prediction of the reduced order model at the previous time step, {which will be the subject of future work}.
\end{remark}

When the framework is used in parallel, with subdomains, the calibration of the error indicator is local to each subdomain, so that the decision of enrichment {in the full domain} during the \textit{online} stage can be {triggered} by a particular subdomain {of interest}.

\section{Numerical applications}
\label{sec:num}

We consider two behavior laws in the numerical applications:
\begin{itemize}
\item[(elas)] isotropic thermal expansion and temperature-dependent cubic elasticity: the behavior law is $\sigma=\mathcal{A}:\left(\epsilon-\epsilon^{\rm th}\right)$, where $\epsilon^{\rm th}=\alpha^{\rm th}\left(T-T_0\right)I$, with $I$ the second-order identity tensor and $\alpha^{\rm th}$ the thermal expansion coefficient in MPa.K${}^{-1}$ depending on the temperature.
The elastic stiffness tensor $\mathcal{A}$ does not depend on the solution $u$ and is defined in Voigt notations by
\begin{equation}
\label{eq:tensorElas}
\mathcal{A}=
\begin{pmatrix} 
y_{1111} & y_{1122} & y_{1122} & 0 & 0 & 0 \\
y_{1122} & y_{1111} & y_{1122} & 0 & 0 & 0 \\
y_{1122} & y_{1122} & y_{1111} & 0 & 0 & 0 \\
0&0 &0 & y_{1212} & 0&0 \\
0&0 & 0 &0 &  y_{1212} &0 \\
0&0  &0 &0 &0 & y_{1212}  \\
\end{pmatrix},
\end{equation}
where the temperature $T$ is given by the thermal loading, $T_0=20^\circ$C is a reference temperature and the coefficients $y_{1111}$, $y_{1122}$ and $y_{1212}$ ({elastic coefficients} in MPa) depend on the temperature. This law does not feature any internal variable to compute.

\item[(evp)] Norton flow with nonlinear kinematic hardening: the elastic part is given by $\sigma=\mathcal{A}:\left(\epsilon-\epsilon^{\rm th}-\epsilon^P\right)$, where $\mathcal{A}$ and $\epsilon^{\rm th}$ are the same as the (elas) law, $\epsilon^P$ is the plastic strain tensor. The viscoplastic part requires solving the system of ODEs:
\begin{equation}
\left\{
\begin{aligned}
\dot{\epsilon}^P&=\dot{p}\sqrt{\frac{3}{2}}\frac{s-\frac{2}{3}C\alpha}{\sqrt{\left(s-\frac{2}{3}C\alpha\right):\left(s-\frac{2}{3}C\alpha\right)}},\\
\dot{\alpha}&=\dot{\epsilon}^P - \dot{p}D\alpha,\\
\dot{p}&=\left\langle\frac{f_r}{K}\right\rangle^m,
\end{aligned}
\right.
\end{equation}
where $p$ is the cumulated plasticity, $f_r=\sqrt{\frac{3}{2}}\sqrt{\left(s-\frac{2}{3}C\alpha\right):\left(s-\frac{2}{3}C\alpha\right)} - R_0$ defines the yield surface, {$\alpha$ (dimensionless) is the internal variable associated to the back-stress tensor $X=\frac{2}{3}C\alpha$  representing the center of the elastic domain in the stress space}, $s:=\sigma-\frac{1}{3}{\rm Tr}(\sigma)I$ (with Tr the trace operator) is the deviatoric component of the stress tensor, and $\langle\cdot\rangle$ denotes the positive part operator. The yield criterion is $f_r\leq 0$.
The hardening material coefficients $C$ (in MPa) and $D$ (dimensionless), the Norton material coefficient $K$ (in MPa.s${}^{\frac{1}{\rm m}}$), the Norton exponential material coefficient $m$ (dimensionless), and the initial yield stress $R_0$ (in MPa) depend on the temperature.
The internal variables considered here are $\epsilon^P$, $\alpha$ and $p$, and the ODE's initial conditions are $\epsilon^P=0$, $\alpha=0$ and $p=0$ at $t=0$.
\end{itemize}

Two test cases are considered: an academic one in Section~\ref{sec:aca} and a high-pressure turbine blade setting of industrial complexity in Section~\ref{sec:indus}.

\subsection{Academic example}
\label{sec:aca}

We consider a simple geometry in the shape of a bow tie, to enforce plastic effects on the tightest area, see Figure~\ref{fig:aca1}. The structure is subjected to different variabilities of the loading (temperature, rotation, pressure), described in Figures~\ref{fig:aca1}-\ref{fig:aca3}. The axis of rotation {is located on the left of the object along the x-axis,} and the pressure field is represented in Figure~\ref{fig:aca1}. {The rotation of the object is not computed: only the inertia effects are modeled in the volumic force term $f$ in~\eqref{eq:r2}}. Four temperature fields are considered, two of them are represented in Figure~\ref{fig:aca2} (``temperature\_field\_1'' is a uniform $20^\circ$C field, {``temperature\_field\_2'' is a 3D Gaussian with a maximum in the thin part of the object, close to an edge}, ``temperature\_field\_3'' is proportional to ``temperature\_field\_2'', ``temperature\_field\_4'' obtained from ``temperature\_field\_2'' by random perturbation of $10\%$ magnitude independently at each point). Notice that the {irregularity} of ``temperature\_field\_4'' will lead to small scaled structures in the cumulated plasticity and stress fields involving this variability. {Notice also that the temperature field are not computed during the simulation: they are loading data for the mechanical computation.} Figure~\ref{fig:aca3} presents the three variabilities considered: \emph{computation 1} and \emph{computation 2} encountered in the \textit{offline} phase, and \emph{new} encountered in the \textit{online} phase. The pressure loading is obtained by multiplying the pressure coefficient by the pressure field represented in Figure~\ref{fig:aca1} (normals on the boundary are directed towards the exterior) and at each time step, the temperature field is obtained by linear interpolation between the previous and following fields in the temporal sequence. {Notice that \emph{computation 1} and \emph{computation 2} are not defined on the same temporal range.}

\begin{figure}[H]
  \centering
 \includegraphics[width=0.6\textwidth]{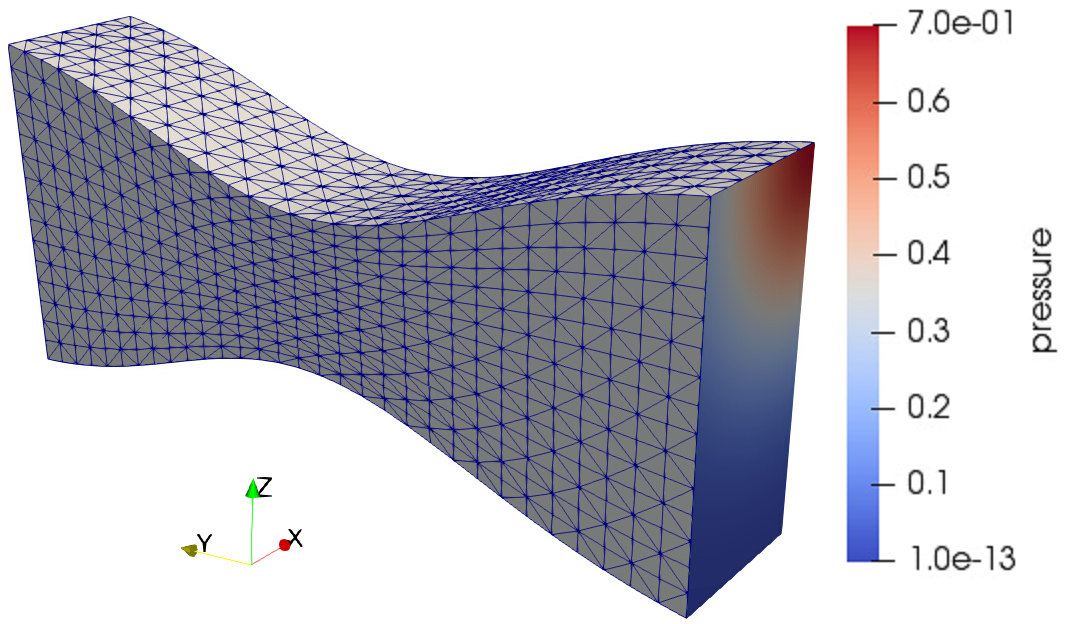}
  \caption{Academic test case: mesh and pressure field represented on its surface of application; {the axis of rotation is located on the left of the object along the x-axis}.}
  \label{fig:aca1}
\end{figure}

\begin{figure}[H]
  \centering
  \begin{minipage}{.49\linewidth}
 \includegraphics[width=\textwidth]{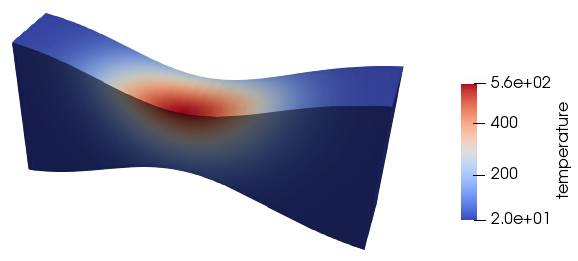}
 \qquad ''temperature\_field\_2''
   \end{minipage} \hfill
   \begin{minipage}{.49\linewidth}
 \includegraphics[width=\textwidth]{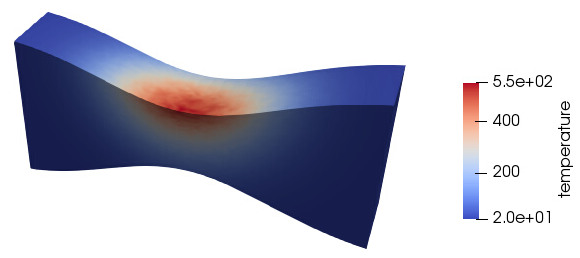}
 \qquad ''temperature\_field\_4''
   \end{minipage}
  \caption{Two different variabilities for the temperature loading (in ${}^\circ$C) used in the academic test case.}
  \label{fig:aca2}
\end{figure}

\setlength\figureheight{4cm}
\setlength\figurewidth{5cm}
\begin{figure}[H]
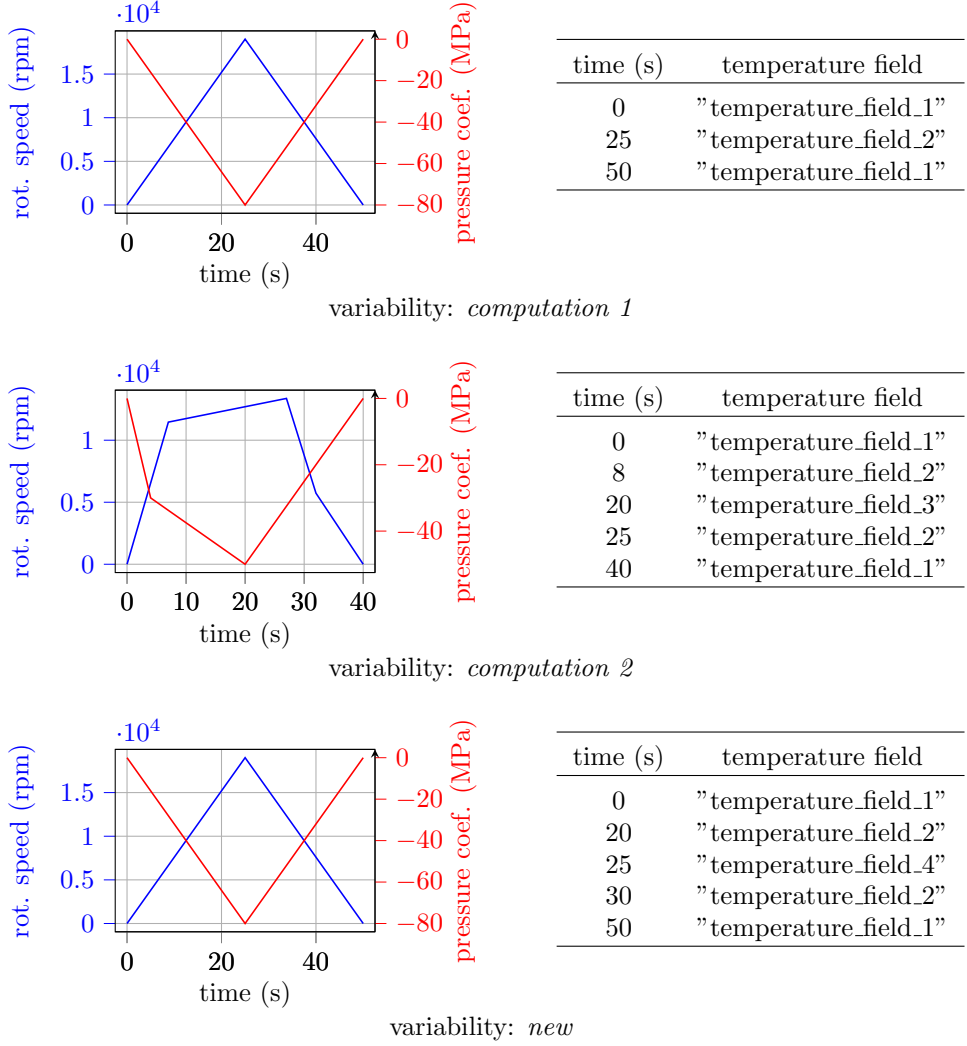

  \centering
\input{rot1.tex}
    \qquad
    \begin{tabular}[b]{cc}\hline &\\[-0.8em]
      time (s) & temperature field \\ \hline &\\[-0.8em]
      0  & ''temperature\_field\_1'' \\ 
      25 & ''temperature\_field\_2'' \\ 
      50 & ''temperature\_field\_1'' \\ \hline
\vspace{1.cm}      
    \end{tabular}
 \\[-0.5em] variability: \emph{computation 1}\\[1em]
    
\input{rot2.tex}
    \qquad
    \begin{tabular}[b]{cc}\hline &\\[-0.8em]
      time (s) & temperature field \\ \hline &\\[-0.8em]
      0  & ''temperature\_field\_1'' \\ 
      8 & ''temperature\_field\_2'' \\ 
      20 & ''temperature\_field\_3'' \\ 
      25 & ''temperature\_field\_2'' \\ 
      40 & ''temperature\_field\_1'' \\ \hline
\vspace{0.5cm}      
    \end{tabular}
 \\[-0.5em] variability: \emph{computation 2}\\[1em]
    
\input{rot1.tex}
    \qquad
    \begin{tabular}[b]{cc}\hline &\\[-0.8em]
      time (s) & temperature field \\ \hline &\\[-0.8em]
      0  & ''temperature\_field\_1'' \\ 
      20 & ''temperature\_field\_2'' \\ 
      25 & ''temperature\_field\_4'' \\ 
      30 & ''temperature\_field\_2'' \\ 
      50 & ''temperature\_field\_1'' \\ \hline
\vspace{.5cm}      
    \end{tabular}
\\[-0.5em] variability: \emph{new}
    
  \caption{Considered loading variabilities for the academic test case; left: rotation speed (\protect\blueline) and pressure coefficient (\protect\redline) with respect to the time; right: temporal sequence for the temperature field.}
  \label{fig:aca3}
  \end{figure}

The characteristics for the academic test cases are given in Table~\ref{tab:aca4}.
\begin{table}[H]
  \centering
\begin{tabular}{|c|c|}
  \hline &\\[-0.8em]
  number of dofs & 78'120\\
  \hline &\\[-0.8em]
  number of (quadratic) tetrahedra & 16'695 \\
  \hline &\\[-0.8em]
  number of integration points & 81'375 \\
  \hline &\\[-0.8em]
  number of time steps &  \emph{computation 1}: 50,  \emph{computation 2}: 40,  \emph{new}: 50\\
  \hline &\\[-0.8em]
 behavior law &evp (Norton flow with nonlinear kinematic hardening \\
  \hline
\end{tabular}
  \caption{Characteristics for the academic test case.}
  \label{tab:aca4}
\end{table}

The correlations between the ROM-Gappy-POD residual $\mathcal{E}$~\eqref{eq:error_ind_1} and the error $E$~\eqref{eq:error_1} on the dual quantities cumulated plasticity $p$ and first component of the stress tensor $\sigma_{11}$ are investigated in Table~\ref{tab:aca5}. The reduced solutions used for $\mathcal{E}$ correspond to the calibration step in the \textit{offline} stage, in the second row of Figure~\ref{fig:workflow_offline}, where we recall that the reduced Newton algorithm~\eqref{eq:reducedNewton} is computed with a large tolerance $\epsilon^{\rm ROM}_{\rm Newton}=0.1$ on the variabilities encountered in the data generation step. For the cumulated plasticity field, the values before the first plastic effects are neglected. A strong correlation appears in all the considered cases, although outliers are observed for the last time steps, where the building of residual stresses at low loadings are more difficult to predict with the ROM.
\setlength\figureheight{3.5cm}
\setlength\figurewidth{3.5cm}
\begin{table}[H]
\centering
\begin{tabular}{>{\centering\arraybackslash}m{6em}|>{\centering\arraybackslash}m{19.25em}|>{\centering\arraybackslash}m{19.25em}|}
\cline{2-3}
 & $p$ & $\sigma_{11}$ \\ \hline
\multicolumn{1}{|>{\centering\arraybackslash}m{6em}|}{}&&\\[-0.5em]
\multicolumn{1}{|>{\centering\arraybackslash}m{6em}|}{computation 1} & \input{academicSansEnrichissementOffComp1OnNewindicator_evrcum_Computation1.tex}\input{academicSansEnrichissementOffComp1OnNewscatter_evrcum.tex}
  & \input{academicSansEnrichissementOffComp1OnNewindicator_sig11_Computation1.tex}\input{academicSansEnrichissementOffComp1OnNewscatter_sig11.tex} \\ \hline
\multicolumn{1}{|>{\centering\arraybackslash}m{6em}|}{}&&\\[-0.5em]
\multicolumn{1}{|>{\centering\arraybackslash}m{6em}|}{computation 2} &  \input{academicSansEnrichissementOffComp1Comp2OnNewindicator_evrcum_Computation2.tex}\input{academicSansEnrichissementOffComp1Comp2OnNewscatter_evrcum.tex} & \input{academicSansEnrichissementOffComp1Comp2OnNewindicator_sig11_Computation2.tex}\input{academicSansEnrichissementOffComp1Comp2OnNewscatter_sig11.tex} \\ \hline
\end{tabular}
\caption{Illustration of the correlation between the ROM-Gappy-POD residual $\mathcal{E}$~\eqref{eq:error_ind_1} and the error $E$~\eqref{eq:error_1} on the dual quantities cumulated plasticity $p$ and first component of the stress tensor $\sigma_{11}$.}
\label{tab:aca5}
\end{table}

We now illustrate the quality of the error indicator~\eqref{eq:indicator}, and its ability to increase the accuracy of the reduced order model when used in an enrichment strategy as described in the workflow illustrated in Figure~\ref{fig:workflow_online}. In Tables~\ref{tab:aca6} and~\ref{tab:aca7}, we compare the error indicator~\eqref{eq:indicator} with the error~\eqref{eq:error_1} for various \textit{offline} and \textit{online} variabilities respectively without and with enrichment of the reduced order model. Although our error indicator is not a certified upper bound, we observe that thanks to the calibration process, its values are in the vast majority larger than the exact error, except in two regimes: (i) when the errors are very large (the calibration has been carried-out for mild errors, since we used the references from the \textit{offline} variabilities and enforced reasonable errors in line 3 of Algorithm~\ref{algo:calibration}), and (ii) sometimes in the last time steps where the residual stresses build up and where we identified outliers in the Gaussian regressor process. In Table~\ref{tab:aca6}, we observe that without enrichment the errors are controlled whenever the \textit{online} variability is contained in the \textit{offline} variability. In the other cases, the error becomes very large, and the ROM prediction becomes useless. In Table~\ref{tab:aca7}, at the times when the ROM is enriched, both the error indicator and the error are set to zero, since the ROM prediction is replaced by a HF solution. The ROM is enriched when the ${\rm Gpr}^p(\mathcal{E}^p)>0.2$ or ${\rm Gpr}^{\sigma_{11}}(\mathcal{E}^{\sigma_{11}})>0.2$. We observe that for cases where the \textit{online} variability is included in the \textit{offline} variability, the errors are still controlled and no enrichment occurs. In the other cases, the enrichment occurs a few times, so that the errors remain controlled below 0.2. For the \textit{online} variability \emph{new}, the ROM is enriched 6 times for an \textit{offline} variability \emph{computation 1} and only 3 times for an \textit{online} variability \emph{computation 1} and  \emph{computation 2}: in the latter case, the initial reduced order basis generates a larger base and needs less enrichment.

\setlength\figureheight{3.5cm}
\setlength\figurewidth{3.2cm}
\begin{table}[H]
\centering
\begin{tabular}{|>{\centering\arraybackslash}m{8.5em}|*{2}{>{\centering\arraybackslash}m{18em}|}}\hline
&&\\[-1.em]
\backslashbox{\textit{online}}{\textit{offline}}
&\makebox[18em]{computation 1}&\makebox[18em]{computation 1 and computation 2}\\\hline
&&\\[-0.5em]
computation 1 &\input{academicSansEnrichissementOffComp1OnComp1indicator_evrcum_Online.tex}\input{academicSansEnrichissementOffComp1OnComp1indicator_sig11_Online.tex}&\input{academicSansEnrichissementOffComp1Comp2OnComp1indicator_evrcum_Online.tex}\input{academicSansEnrichissementOffComp1Comp2OnComp1indicator_sig11_Online.tex}\\\hline
&&\\[-0.5em]
computation 2 &\input{academicSansEnrichissementOffComp1OnComp2indicator_evrcum_Online.tex}\input{academicSansEnrichissementOffComp1OnComp2indicator_sig11_Online.tex}&\input{academicSansEnrichissementOffComp1Comp2OnComp2indicator_evrcum_Online.tex}\input{academicSansEnrichissementOffComp1Comp2OnComp2indicator_sig11_Online.tex}\\\hline
&&\\[-0.5em]
new &\input{academicSansEnrichissementOffComp1OnNewindicator_evrcum_Online.tex}\input{academicSansEnrichissementOffComp1OnNewindicator_sig11_Online.tex}&\input{academicSansEnrichissementOffComp1Comp2OnNewindicator_evrcum_Online.tex}\input{academicSansEnrichissementOffComp1Comp2OnNewindicator_sig11_Online.tex}\\\hline
\end{tabular}
\caption{Comparison of the error indicator~\eqref{eq:indicator} with the error~\eqref{eq:error_1} for various \textit{offline} and \textit{online} variabilities, without enrichment of the reduced order model. The category ``\textit{offline}'' for the columns refers to the variabilities used in the data generation step of the \textit{offline} stage, whereas the category ``\textit{online}'' for the rows refers to the variability considered in the \textit{online} stage.}
\label{tab:aca6}
\end{table}

\setlength\figureheight{3.5cm}
\setlength\figurewidth{3.2cm}
\begin{table}[H]
\centering
\begin{tabular}{|>{\centering\arraybackslash}m{8.5em}|*{2}{>{\centering\arraybackslash}m{18em}|}}\hline
&&\\[-1.em]
\backslashbox{\textit{online}}{\textit{offline}}
&\makebox[18em]{computation 1}&\makebox[18em]{computation 1 and computation 2}\\\hline
&&\\[-0.5em]
computation 1 &\input{academicAvecEnrichissementOffComp1OnComp1indicator_evrcum_Online.tex}\input{academicAvecEnrichissementOffComp1OnComp1indicator_sig11_Online.tex}&\input{academicAvecEnrichissementOffComp1Comp2OnComp1indicator_evrcum_Online.tex}\input{academicAvecEnrichissementOffComp1Comp2OnComp1indicator_sig11_Online.tex}\\\hline
&&\\[-0.5em]
computation 2 &\input{academicAvecEnrichissementOffComp1OnComp2indicator_evrcum_Online.tex}\input{academicAvecEnrichissementOffComp1OnComp2indicator_sig11_Online.tex}&\input{academicAvecEnrichissementOffComp1Comp2OnComp2indicator_evrcum_Online.tex}\input{academicAvecEnrichissementOffComp1Comp2OnComp2indicator_sig11_Online.tex}\\\hline
&&\\[-0.5em]
new &\input{academicAvecEnrichissementOffComp1OnNewindicator_evrcum_Online.tex}\input{academicAvecEnrichissementOffComp1OnNewindicator_sig11_Online.tex}&\input{academicAvecEnrichissementOffComp1Comp2OnNewindicator_evrcum_Online.tex}\input{academicAvecEnrichissementOffComp1Comp2OnNewindicator_sig11_Online.tex}\\\hline
\end{tabular}
\caption{Comparison of the error indicator~\eqref{eq:indicator} with the error~\eqref{eq:error_1} for various \textit{offline} and \textit{online} variabilities, with enrichment of the reduced order model.}
\label{tab:aca7}
\end{table}

We now compare the reference HF prediction of the considered \textit{online} variability with the ROM prediction without and with enrichment, in a case where this \textit{online} variability is included in the \textit{offline} variability (Figure~\ref{fig:acaRes}) and in a case where it is not included (Figure~\ref{fig:acaRes2}). {In Figures~\ref{fig:acaRes} and~\ref{fig:acaRes2},  dual quantities with index ``ref." refers to the HF reference at the considered \textit{offline} variability, ``nores.'' to the ROM without enrichment and the absence of index to the ROM with enrichment.} In the first case, the ROM predictions with and without enrichment are accurate (the magnitude of $\sigma_{11}$ is small with respect to the ones of $\sigma_{22}$, so that the small differences observed in the second plot of Figure~\ref{fig:acaRes} are very small with respect to the magnitude of the tensor $\sigma$). In the second case, the ROM without enrichment leads to large errors, whereas the enrichment allows a good accuracy. We notice that due to the particular profile of the temperature loading ``temperature\_field\_4" (c.f. Figure~\ref{fig:aca2}), the field $\sigma_{11}$ is irregular. Even in such an unfavorable case, only 3 enrichment steps by HFM solutions allows a good accuracy for the ROM.
\setlength\figureheight{3.7cm}
\setlength\figurewidth{5cm}
\begin{figure}[H]
  \centering
 \includegraphics[width=0.49\textwidth]{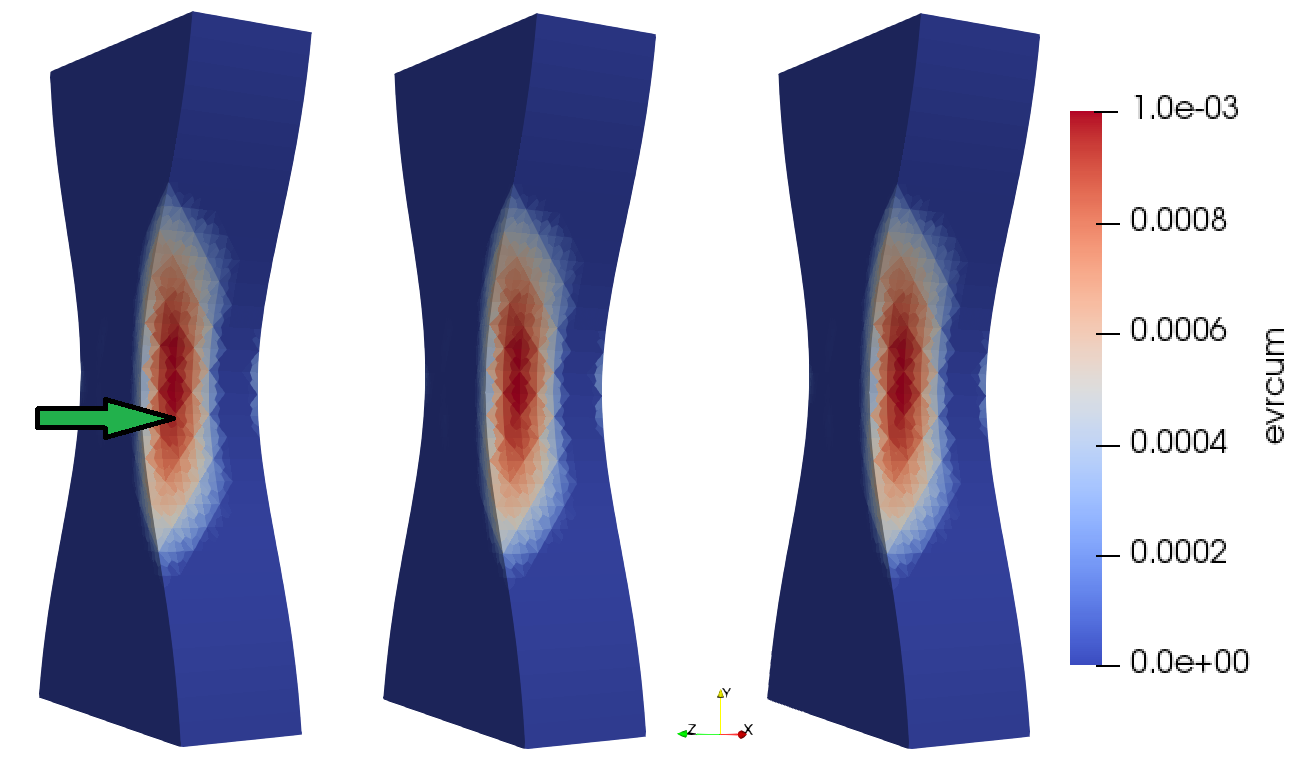}
 \includegraphics[width=0.49\textwidth]{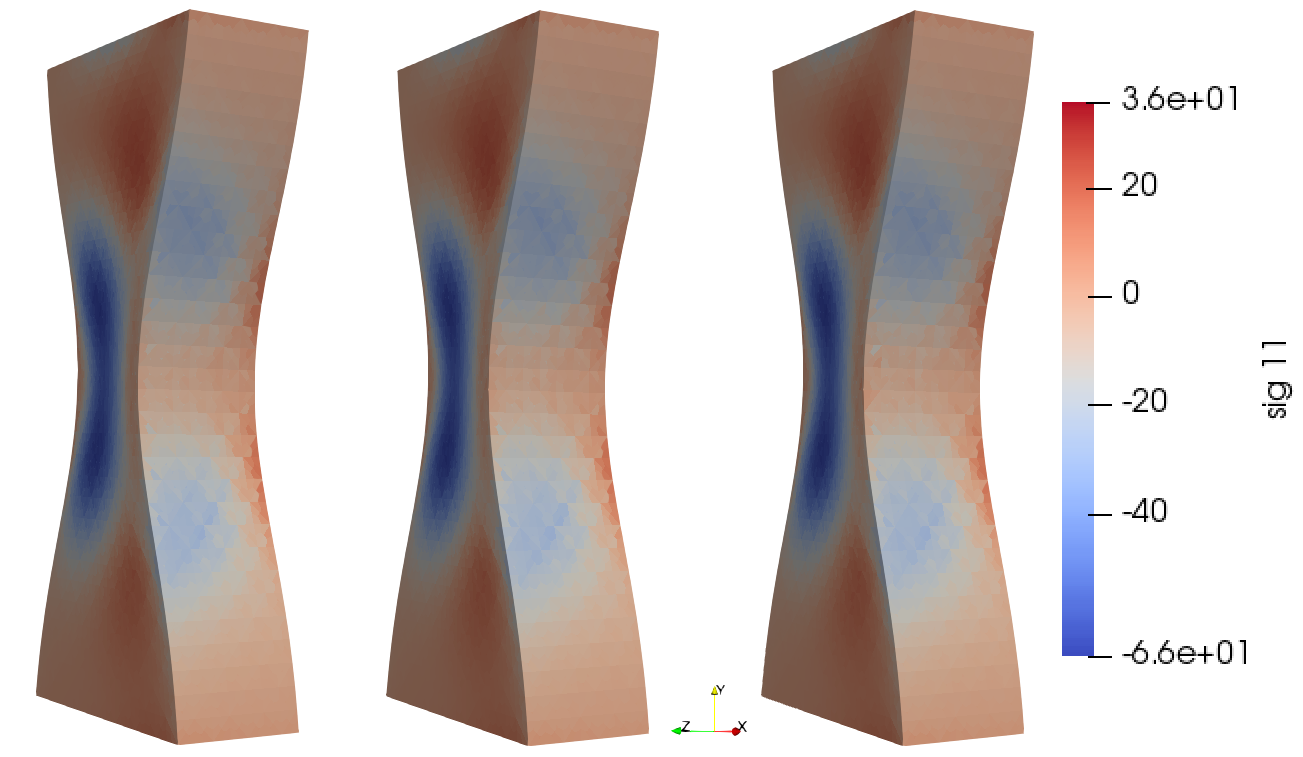}\\
\begin{minipage}{.09\linewidth}\end{minipage}\hfill\begin{minipage}{.11\linewidth}$p_{\rm ref.}$\end{minipage}\hfill\begin{minipage}{.11\linewidth}$\tilde{p}_{\rm nores.}$\end{minipage}\hfill\begin{minipage}{.11\linewidth}$\tilde{p}$\end{minipage}\hfill\begin{minipage}{.14\linewidth}\end{minipage}\hfill
\begin{minipage}{.11\linewidth}$\sigma_{11, {\rm ref.}}$\end{minipage}\hfill\begin{minipage}{.11\linewidth}$\tilde{\sigma}_{11, {\rm nores.}}$\end{minipage}\hfill
\begin{minipage}{.11\linewidth}$\tilde{\sigma}_{11}$\end{minipage}\hfil\begin{minipage}{.09\linewidth}\end{minipage}\\
  \input{OffComp1Comp2OnComp1evrcum.tex}
  \input{OffComp1Comp2OnComp1sig11.tex}
  \input{OffComp1Comp2OnComp1sig22.tex}
  \caption{\textit{offline} variability: \emph{computation 1} and \emph{computation 2}; \textit{online} variability: \emph{computation 1}. Top: representation of dual fields for the reference HF prediction of the \textit{online} variability, the ROM without enrichment, and the ROM with enrichment (left: $p$ at $t=50$s; right: $\sigma_{11}$ at $t=25$s); bottom: comparison of $p$, $\sigma_{11}$ and $\sigma_{22}$ at the point identified by the green arrow on the top-left picture.}
  \label{fig:acaRes}
\end{figure}

\setlength\figureheight{3.7cm}
\setlength\figurewidth{5cm}
\begin{figure}[H]
  \centering
 \includegraphics[width=0.49\textwidth]{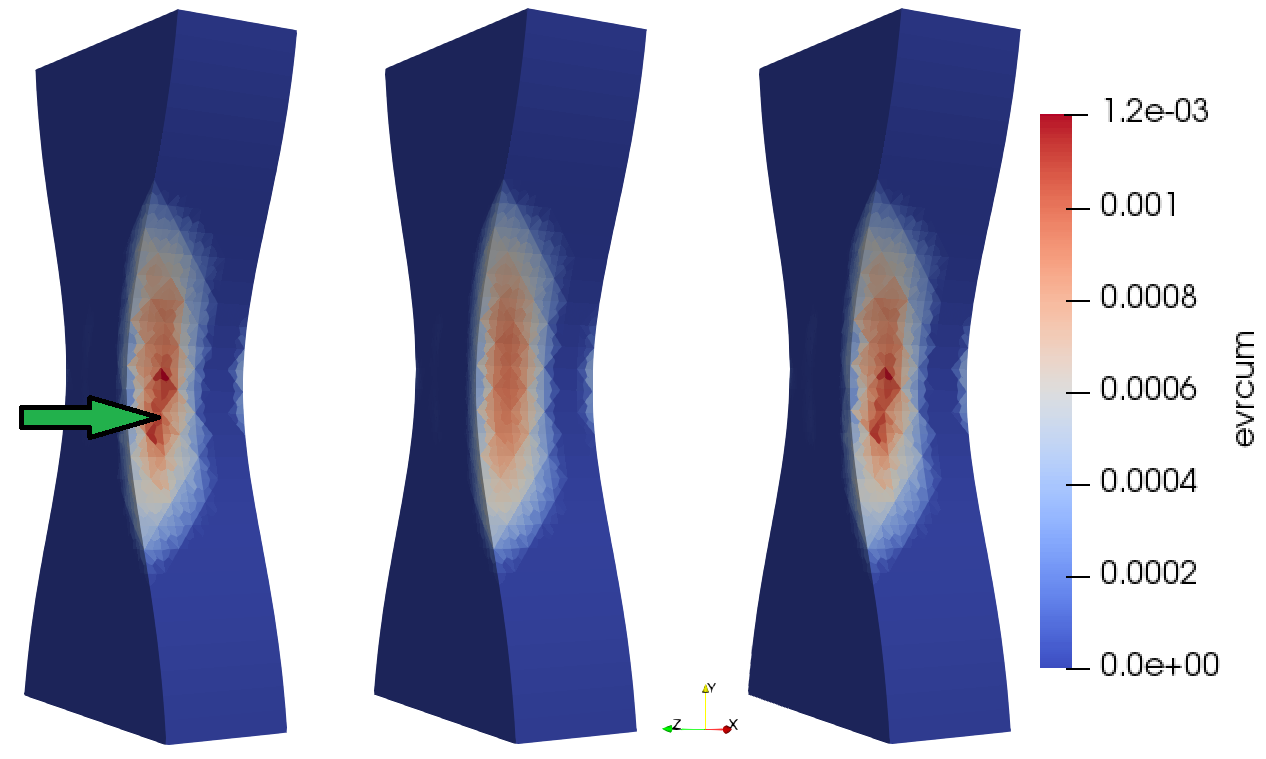}
 \includegraphics[width=0.49\textwidth]{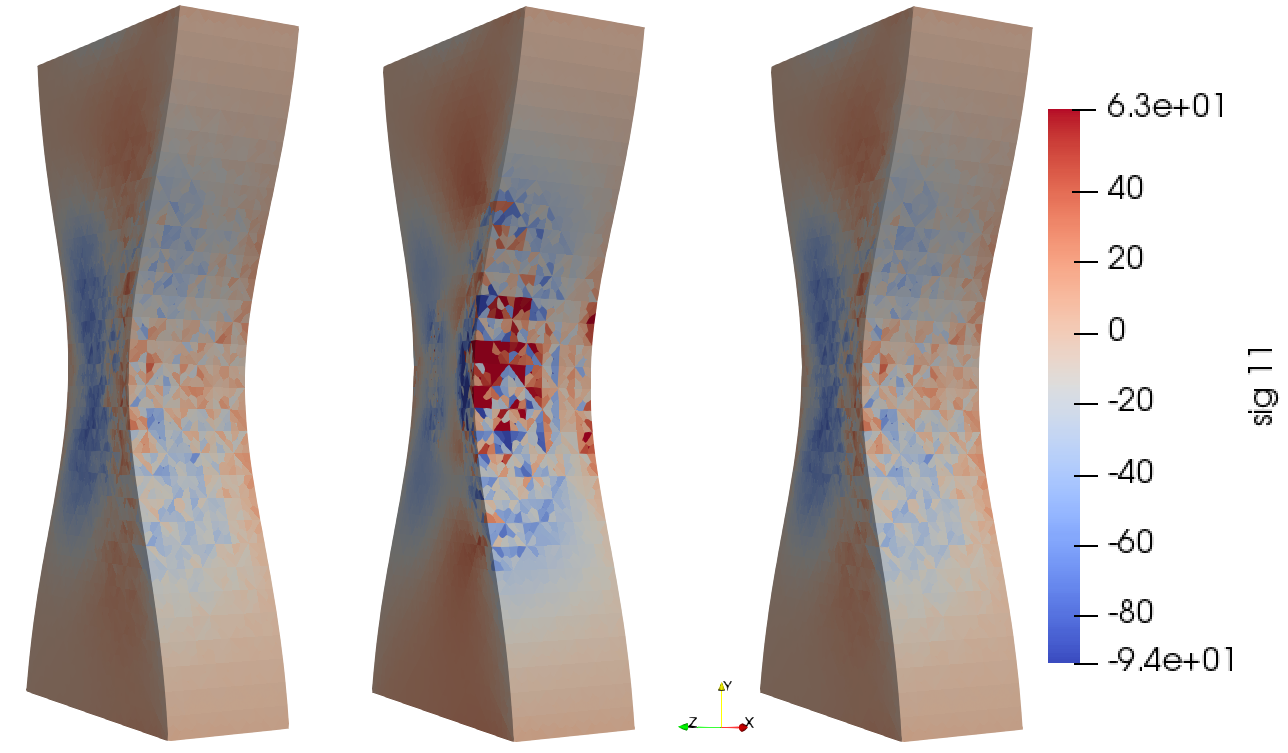}\\
\begin{minipage}{.09\linewidth}\end{minipage}\hfill\begin{minipage}{.11\linewidth}$p_{\rm ref.}$\end{minipage}\hfill\begin{minipage}{.11\linewidth}$\tilde{p}_{\rm nores.}$\end{minipage}\hfill\begin{minipage}{.11\linewidth}$\tilde{p}$\end{minipage}\hfill\begin{minipage}{.14\linewidth}\end{minipage}\hfill
\begin{minipage}{.11\linewidth}$\sigma_{11, {\rm ref.}}$\end{minipage}\hfill\begin{minipage}{.11\linewidth}$\tilde{\sigma}_{11, {\rm nores.}}$\end{minipage}\hfill
\begin{minipage}{.11\linewidth}$\tilde{\sigma}_{11}$\end{minipage}\hfil\begin{minipage}{.09\linewidth}\end{minipage}\\
  \input{OffComp1Comp2OnNewevrcum.tex}
  \input{OffComp1Comp2OnNewsig11.tex}
  \input{OffComp1Comp2OnNewsig22.tex} 
 \caption{\textit{offline} variability: \emph{computation 1} and \emph{computation 2}; \textit{online} variability: \emph{new}. Top: representation of dual fields for the reference HF prediction of the \textit{online} variability, the ROM without enrichment, and the ROM with enrichment (left: $p$ at $t=50$s; right: $\sigma_{11}$ at $t=25$s); bottom: comparison of $p$, $\sigma_{11}$ and $\sigma_{22}$ at the point identified by the green arrow on the top-left picture.}
  \label{fig:acaRes2}
\end{figure}

\subsection{High-pressure turbine blade}
\label{sec:indus}

We consider a simplified geometry of high-pressure turbine blade, featuring four internal cooling channels, introduced in~\cite{ijnme}. The lower part of the blade, referred as the foot, is modeled by an elastic material (we are not interested in predicting the plastic effects in this zone since it does not affect the blade's lifetime) whereas the upper part is modeled by an elastoviscoplastic law. The HFM is computed in parallel using Zset~\cite{zset} with an Adaptive MultiPreconditioned FETI solver~\cite{bovet2017}, see Figure~\ref{fig:subdomainsBlade}.

\begin{center}
\def\svgwidth{\textwidth}
\begin{figure}[H]
  \centering
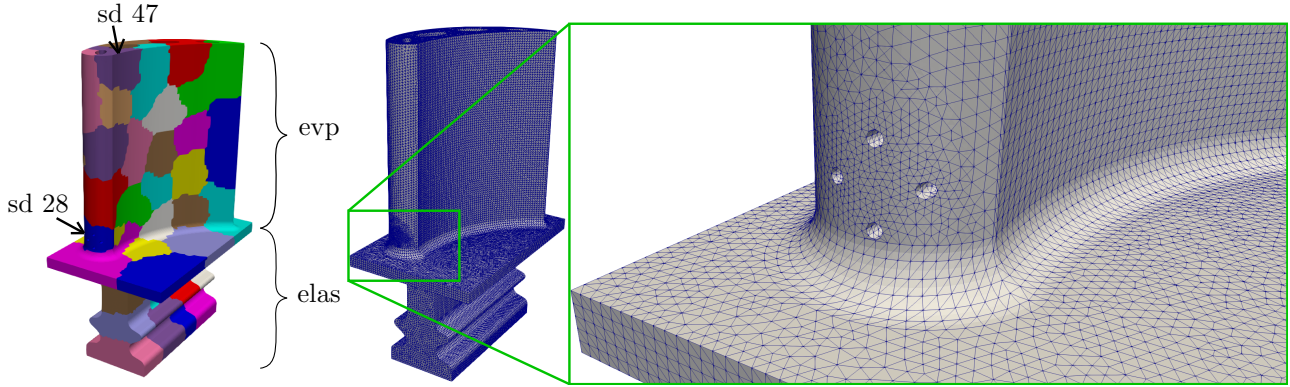
  \caption{a) structure split in 48 subdomains - the top part of the blade’s material is modeled by an elastoviscoplastic law and
the foot’s one by an elastic law, b) mesh for the high-pressure turbine blade with a zoom around the cooling channels.}
\label{fig:subdomainsBlade}
\end{figure}
\end{center}

The loading is different from the application of~\cite{ijnme} and is represented in Figure~\ref{fig:bladeLoading}: 10 temperature fields are considered, the coolest are applied for the lowest rotation speeds, whereas the hottest are applied for the highest rotation speeds. The \textit{online} variability differs from the \textit{offline} variability during the three time steps located around the last three maxima of the rotation speed profile, where only the temperature fields change as indicated by the two pictures at the right side of Figure~\ref{fig:bladeLoading}: the maximum of the temperature is moved from the center to the front of the top part of the blade. As we will see, this local modification will lead to large errors for the ROM if no enrichment strategy is considered.

\setlength\figureheight{5cm}
\setlength\figurewidth{5cm}
\begin{figure}[H]
  \centering
  \begin{minipage}{.37\linewidth}
\input{rot3.tex}
\vspace{0.cm}
   \end{minipage} \hfill
   \begin{minipage}{.62\linewidth}
\includegraphics[width=\textwidth]{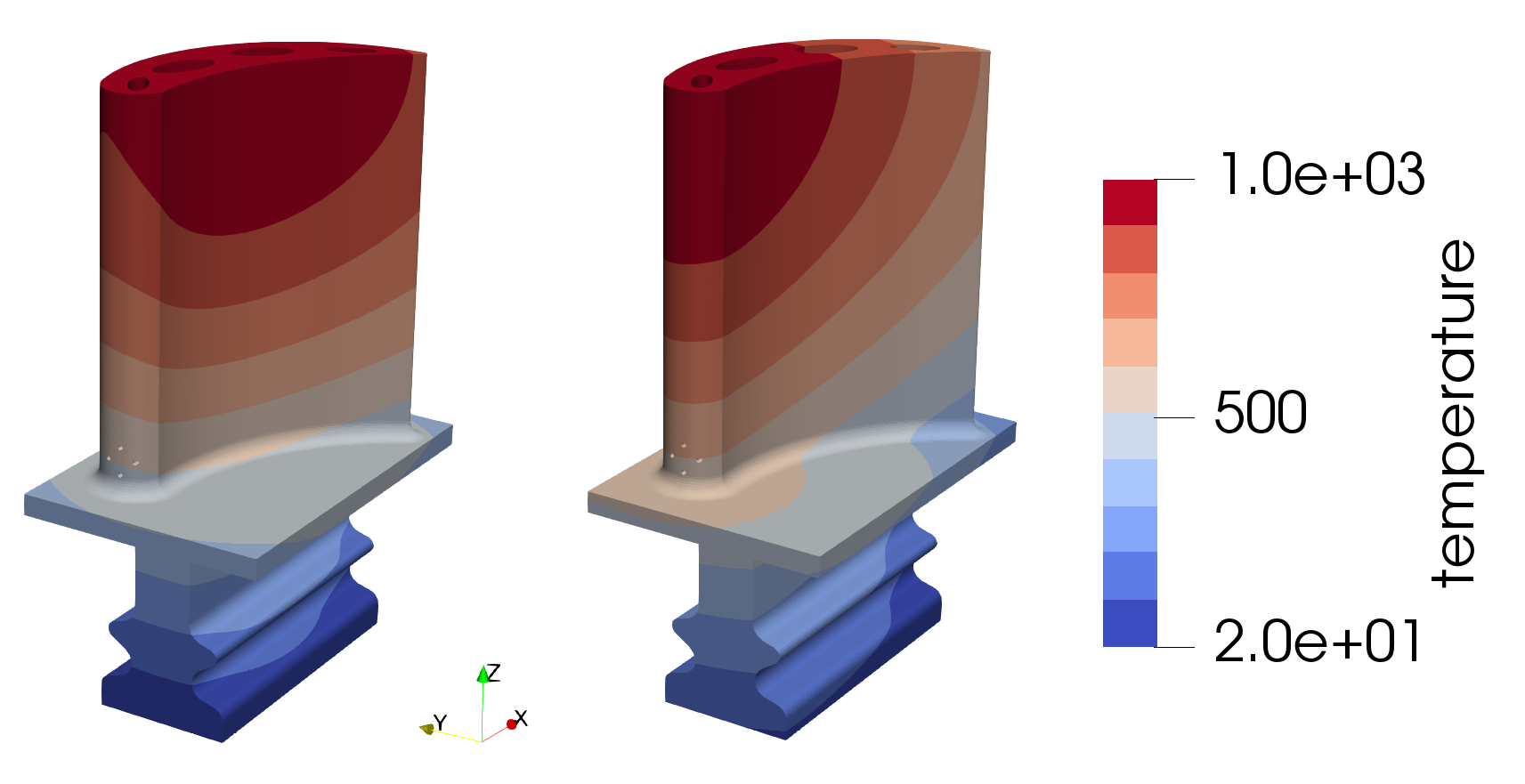}\\
\textit{offline} variability\qquad\quad \textit{online} variability
   \end{minipage}
  \caption{High-pressure turbine test case: left) rotation speed with respect to time; right) representation of maximum temperature fields used in the \textit{offline} and \textit{online} computations; the axis of rotation is located below the blade along the $x$-axis.}
  \label{fig:bladeLoading}
\end{figure}

The characteristics for the high pressure turbine blade case are given in Table~\ref{tab:bladecase}.
\begin{table}[H]
  \centering
\begin{tabular}{|c|c|}
  \hline &\\[-0.8em]
  number of dofs & 4,892'463 \\
  \hline &\\[-0.8em]
  number of (quadratic) tetrahedra & 1'136'732 \\
  \hline &\\[-0.8em]
  number of integration points & 5'683'660\\
  \hline &\\[-0.8em]
  number of time steps & 50\\
  \hline &\\[-0.8em]
 behavior law for the foot &\begin{tabular}{@{}c@{}}elas (temperature-dependent cubic elasticity\\and isotropic thermal expansion)\end{tabular} \\
  \hline &\\[-0.8em]
 behavior law for the blade &evp (Norton flow with nonlinear kinematic hardening \\
  \hline
\end{tabular}
  \caption{Characteristics for the high-pressure turbine blade test case.}
  \label{tab:bladecase}
\end{table}

The computation procedure is presented in Table~\ref{tab:bladecomputation}, all steps being computed in parallel with distributed memory, using MPI for the interprocess communications (48 processors within 2 nodes). The visualization is also parallel with distributed memory using a parallel version of Paraview~\cite{paraview1, paraview2}. 

\begin{table}
  \centering
\begin{tabular}{|c|c|}
  \hline&\\[-0.8em]
 step & algorithm \\
  \hline &\\[-0.8em]
 Data generation & AMPFETI solver in Zset, $\epsilon^{\rm HFM}_{\rm Newton}=10^{-5}$ \\
  \hline &\\[-0.5em]
 Data compression & Distributed Snapshot POD, $\epsilon_{\rm POD}=10^{-5}$ \\
 \hline &\\[-0.8em]
 Operator compression & Distributed NonNegative Orthogonal Matching Pursuit, $\epsilon_{\rm Op. comp.}=10^{-4}$ \\
 \hline &\\[-0.8em]
 Reduced order model & $\epsilon^{\rm ROM}_{\rm Newton}=10^{-4}$ \\
 \hline &\\[-0.5em]
 Dual quantities reconstruction & Distributed Gappy-POD, $\epsilon_{\rm Gappy-POD}=10^{-5}$ \\
  \hline\end{tabular}
  \caption{Description of the computational procedure.}
  \label{tab:bladecomputation}
\end{table}

The correlations between the ROM-Gappy-POD residual $\mathcal{E}$~\eqref{eq:error_ind_1} and the error $E$~\eqref{eq:error_1} on the dual quantities cumulated plasticity $p$ and stress tensor $\sigma$ are investigated in Table~\ref{tab:blade}. This time, we carry-out the calibration process independently on each subdomain. The same conclusion as the academic test cases can be drawn for the correlations between the ROM-Gappy-POD residual $\mathcal{E}$ and the error $E$ on the subdomains 28 and 47 (see Figure~\ref{fig:subdomainsBlade} for the localization of these subdomains).

\setlength\figureheight{3.5cm}
\setlength\figurewidth{3.3cm}
\begin{table}[H]
\centering
\begin{tabular}{>{\centering\arraybackslash}m{6em}|>{\centering\arraybackslash}m{19.25em}|>{\centering\arraybackslash}m{19.25em}|}
\cline{2-3}
 & $p$ & $\sigma_{xx}$ \\ \hline
\multicolumn{1}{|>{\centering\arraybackslash}m{6em}|}{}&&\\[-0.5em]
\multicolumn{1}{|>{\centering\arraybackslash}m{6em}|}{subdomain 28} & \input{indicator_evrcum_Computation1-028.tex}\input{scatter_evrcum-028.tex}
  & \input{indicator_sig22_Computation1-028.tex}\input{scatter_sig22-028.tex} \\ \hline
\multicolumn{1}{|>{\centering\arraybackslash}m{6em}|}{}&&\\[-0.5em]
\multicolumn{1}{|>{\centering\arraybackslash}m{6em}|}{subdomain 47} &  \input{indicator_evrcum_Computation1-047.tex}\input{scatter_evrcum-047.tex} & \input{indicator_sig11_Computation1-047.tex}\input{scatter_sig11-047.tex} \\ \hline
\end{tabular}
\caption{Illustration of the correlation between the ROM-Gappy-POD residual $\mathcal{E}$~\eqref{eq:error_ind_1} and the error $E$~\eqref{eq:error_1} on the dual quantities cumulated plasticity $p$ and a component of the stress tensor $\sigma$.}
\label{tab:blade}
\end{table}

In Table~\ref{tab:blade1}, we compare the error indicator~\eqref{eq:indicator} with the error~\eqref{eq:error_1} for the considered \textit{offline} and \textit{online} variabilities. As for the academic test cases, the values of the error indicator are larger than the error except for very large errors (for which the ROM is useless), and sometimes in the last time steps, as residual forces build up. Without enrichment, the ROM makes very large error. We observe that the subdomain for which the enrichment criterion is used enables to control the error on the corresponding subdomain, whereas the error is larger in the other subdomain. This illustrates that local (in space) quantities of interest can be considered to prevent the enrichment steps to occur too often when it's not needed.

\setlength\figureheight{3.5cm}
\setlength\figurewidth{3.3cm}
\begin{table}[H]
\centering
\begin{tabular}{|>{\centering\arraybackslash}m{8.5em}|*{2}{>{\centering\arraybackslash}m{18em}|}}\hline
&&\\[-1.em]
\backslashbox{enrichment}{plot}
&\makebox[18em]{subdomain 28}&\makebox[18em]{subdomain 47}\\\hline
&&\\[-0.5em]
no enrichment &\input{norestart_indicator_evrcum_Online-028.tex}\input{norestart_indicator_sig22_Online-028.tex}&\input{norestart_indicator_evrcum_Online-047.tex}\input{norestart_indicator_sig11_Online-047.tex}\\\hline
&&\\[-0.5em]
monitoring subdomain 28 &\input{sd28_indicator_evrcum_Online-028.tex}\input{sd28_indicator_sig22_Online-028.tex}&\input{sd28_indicator_evrcum_Online-047.tex}\input{sd28_indicator_sig11_Online-047.tex}\\\hline
&&\\[-0.5em]
monitoring subdomain 47 &\input{sd47_indicator_evrcum_Online-028.tex}\input{sd47_indicator_sig22_Online-028.tex}&\input{sd47_indicator_evrcum_Online-047.tex}\input{sd47_indicator_sig11_Online-047.tex}\\\hline
\end{tabular}
\caption{Comparison of the error indicator~\eqref{eq:indicator} with the error~\eqref{eq:error_1} for the considered \textit{offline} and \textit{online} variabilities. The category ``plot'' for the columns refers to the subdomain for which the error indicator and the error are plotted, whereas the category ``enrichment'' for the rows refers to the subdomain of whom the indicator is used to decide the enrichment step.}
\label{tab:blade1}
\end{table}

In Figures~\ref{fig:bladeRes1} and~\ref{fig:bladeRes2} are illustrated various predictions of dual quantities: the index ``off.'' refers to the HF prediction for the \textit{offline} variability, ``ref.'' to the HF reference for the \textit{online} variability, ``nores.'' to the ROM without enrichment, ``sd28'' to the ROM with enrichment while monitoring the error indicator on subdomain 28, and ``sd47'' to the ROM with enrichment while monitoring the error indicator on subdomain 47. We observe that without enrichment, the ROM suffers from large errors. With enrichment, the monitored subdomain enjoys an accurate ROM prediction. Particularly in Figure~\ref{fig:bladeRes2}, the conclusions hold when the HF reference for the \textit{online} variability is visually different from the HF prediction for the \textit{offline} variability.

\setlength\figureheight{4.3cm}
\setlength\figurewidth{6cm}
\begin{figure}[H]
  \centering
\def\svgwidth{0.48\textwidth}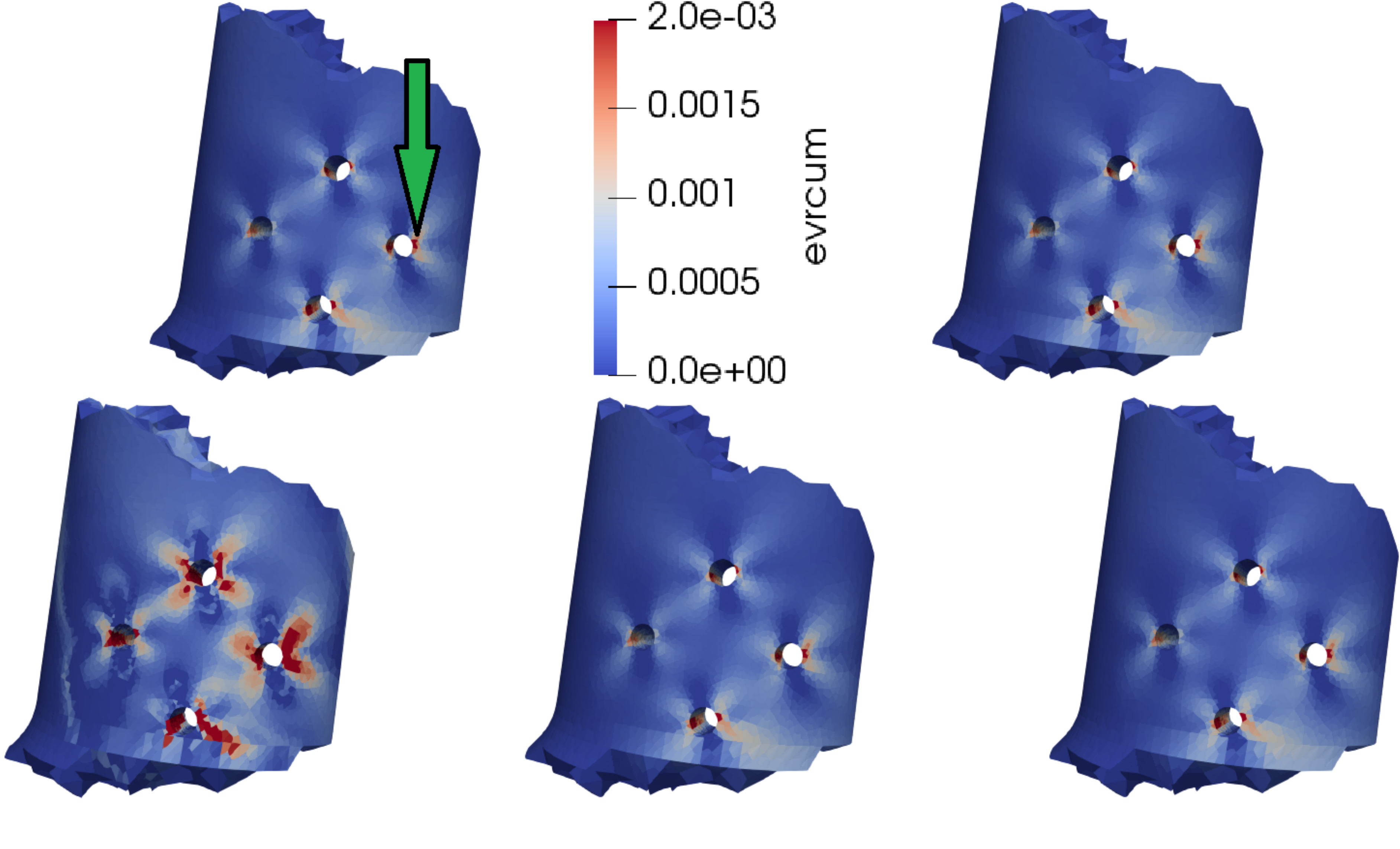
\hspace{0.02\textwidth}
\def\svgwidth{0.48\textwidth}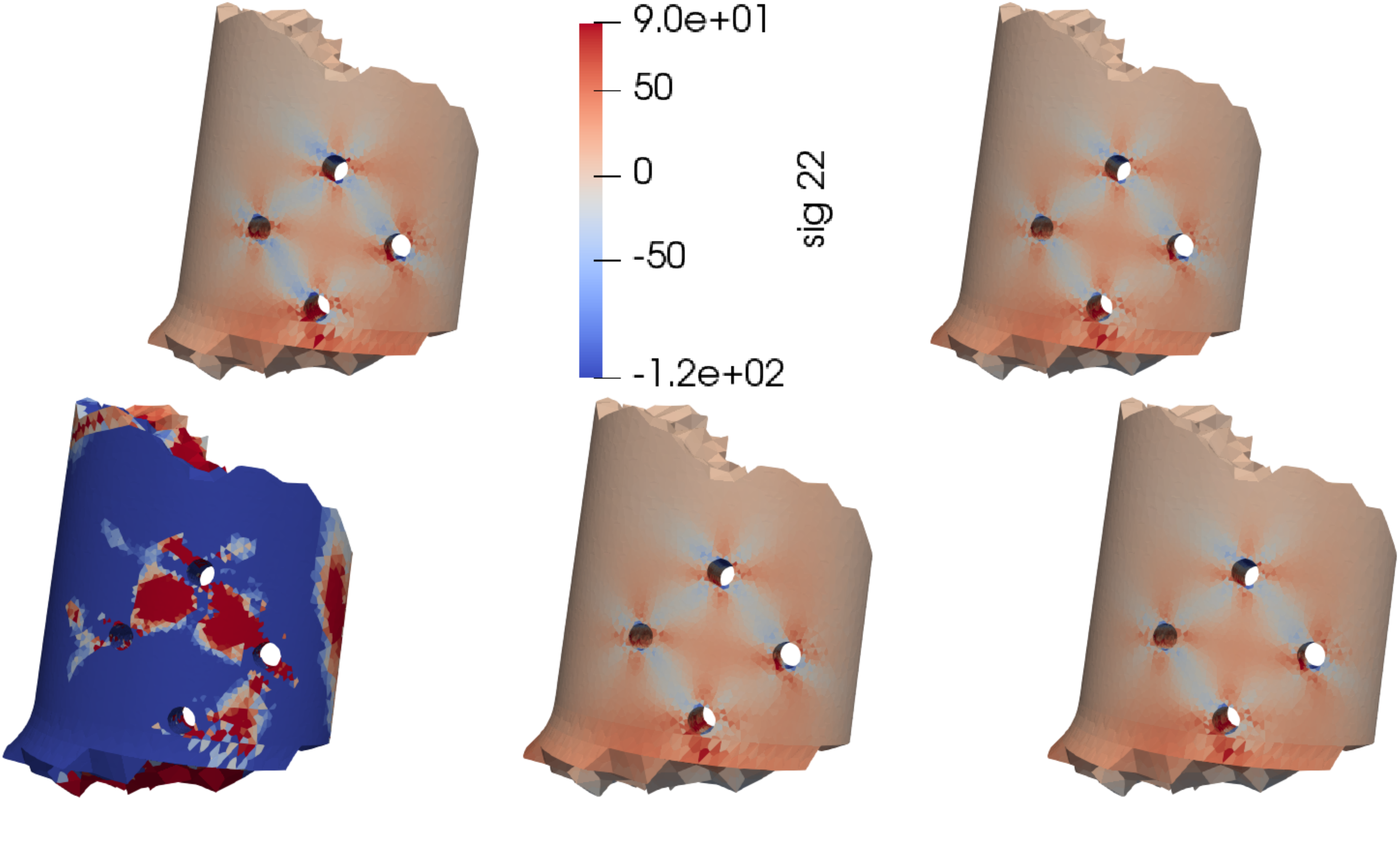
\\[1em]
  \input{evrcum_028.tex}
  \input{sig22_028.tex}
  \caption{Top: diverse HF and ROM dual quantity fields at $t=43.5s$ for subdomain 28: left $p$, right $\sigma_{22}$; bottom: comparison at the point identified by the green arrow on the top-left picture. The components of the stress tensor are in MPa.}
\label{fig:bladeRes1}
\end{figure}

\begin{figure}[H]
  \centering
\def\svgwidth{0.48\textwidth}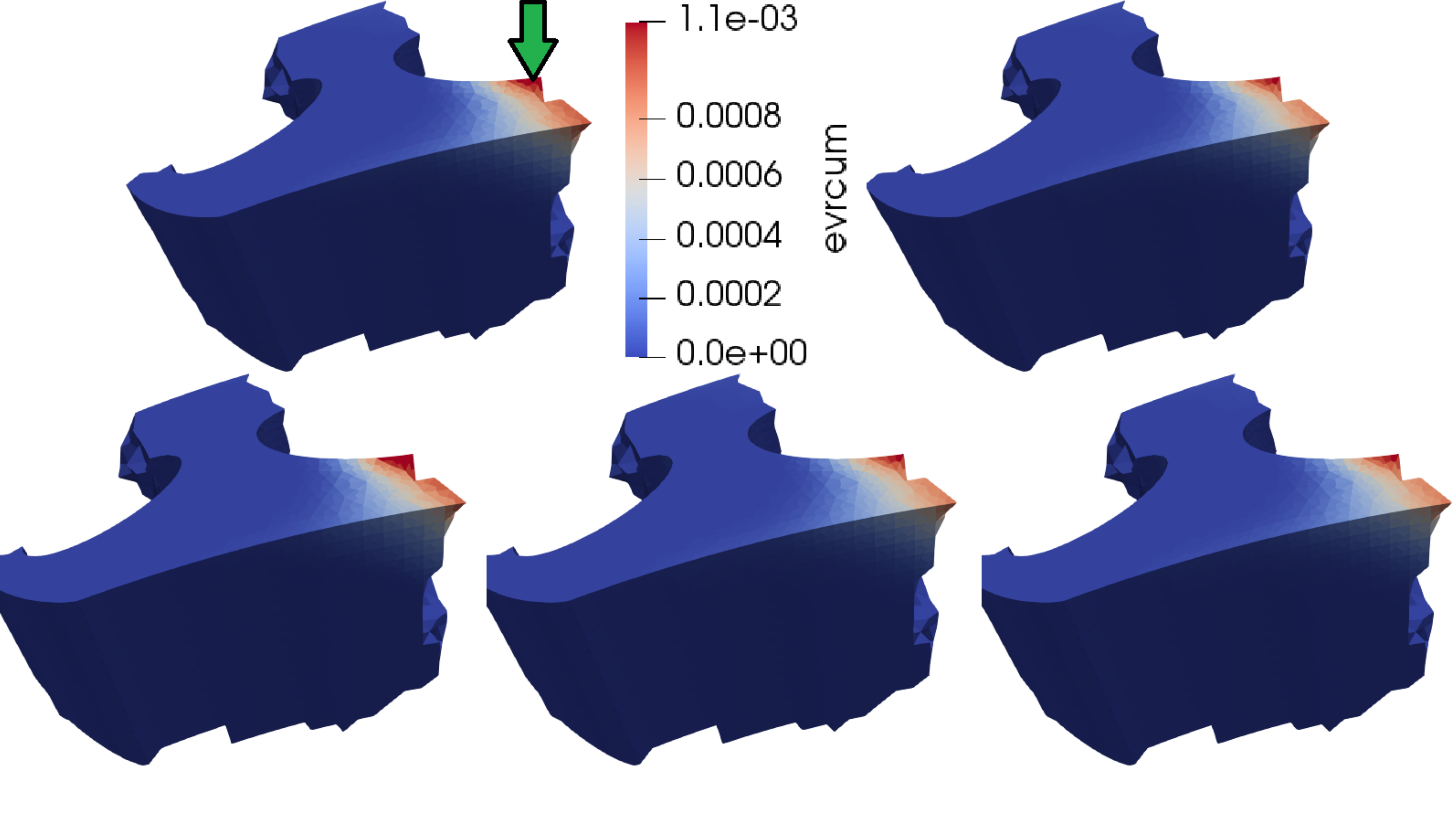
\hspace{0.02\textwidth}
\def\svgwidth{0.48\textwidth}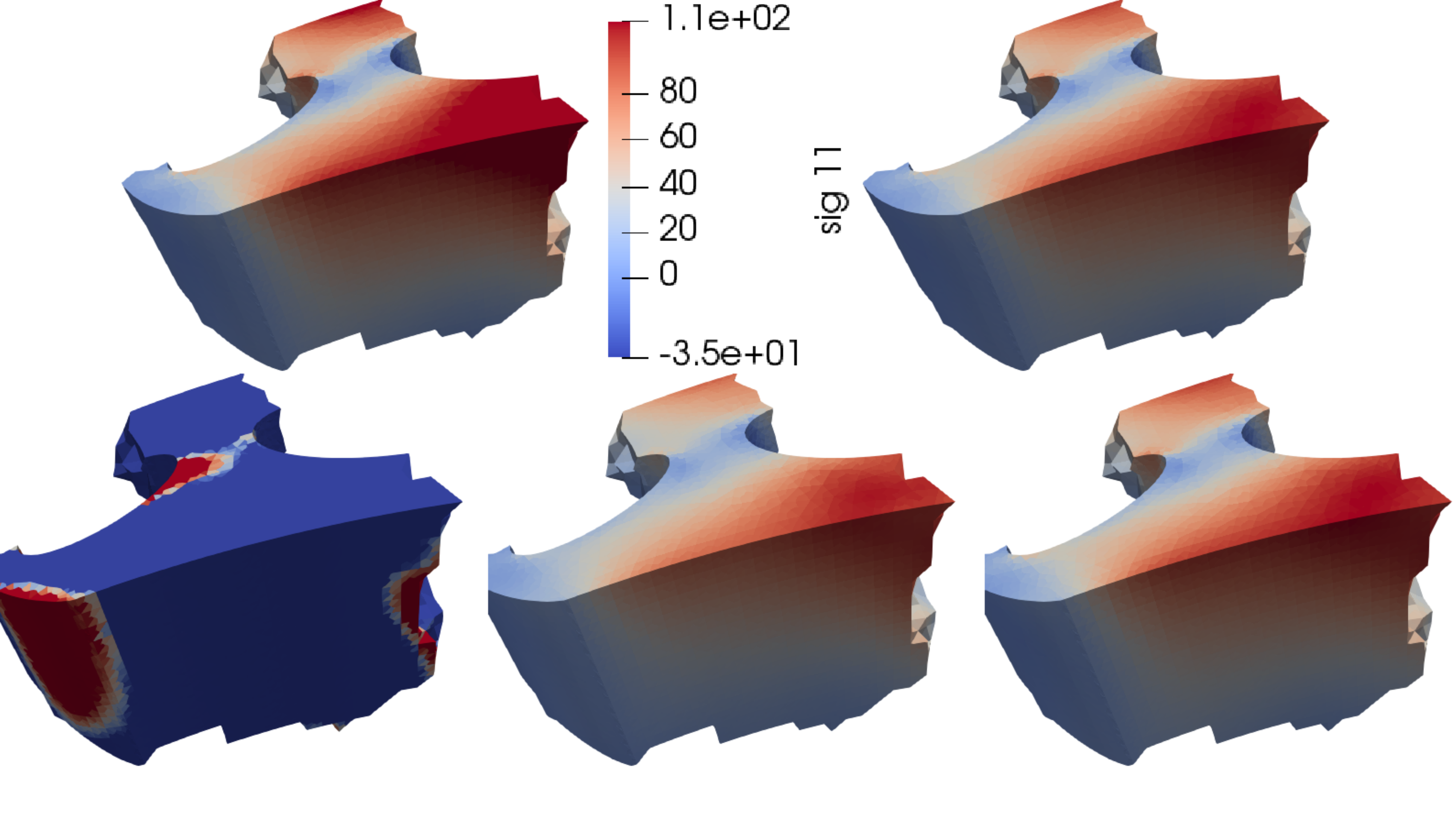
\\[1em]
  \input{evrcum_047.tex}
  \input{sig11_047.tex}
  \caption{Top: diverse HF and ROM dual quantity fields at $t=43.5s$ for subdomain 47: left $p$, right $\sigma_{11}$; bottom: comparison at the point identified by the green arrow on the top-left picture. The components of the stress tensor are in MPa.}
\label{fig:bladeRes2}
\end{figure}

Finally, we represent various predictions of dual quantities on the complete structure in Figure~\ref{fig:bladeRes3}. The ROM without enrichment shows a cumulated plasticity with large errors around the cooling channel, whereas the stress prediction has large errors on the complete structure.

\begin{figure}[H]
  \centering
\def\svgwidth{0.48\textwidth}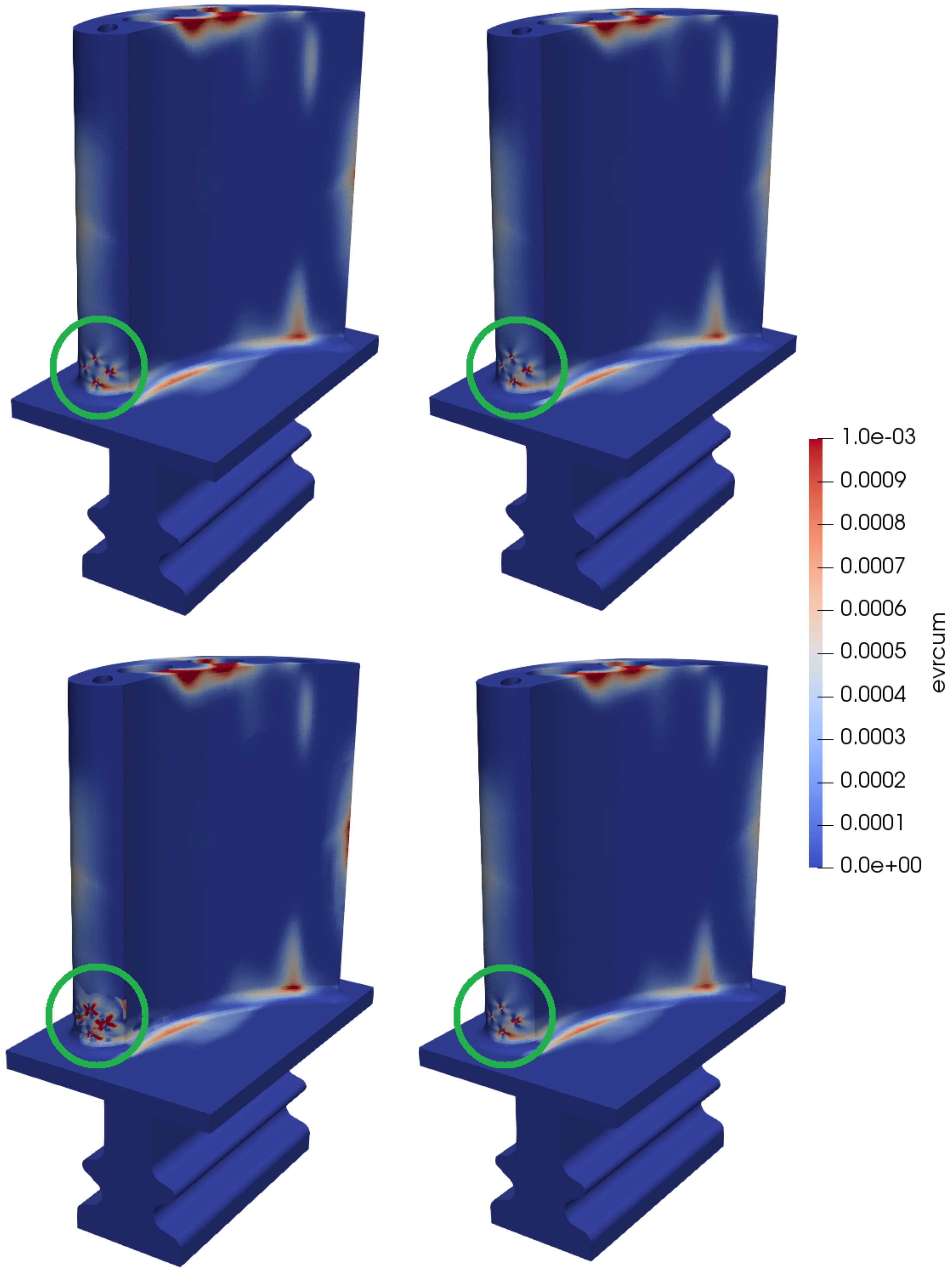
\hspace{0.02\textwidth}
\def\svgwidth{0.48\textwidth}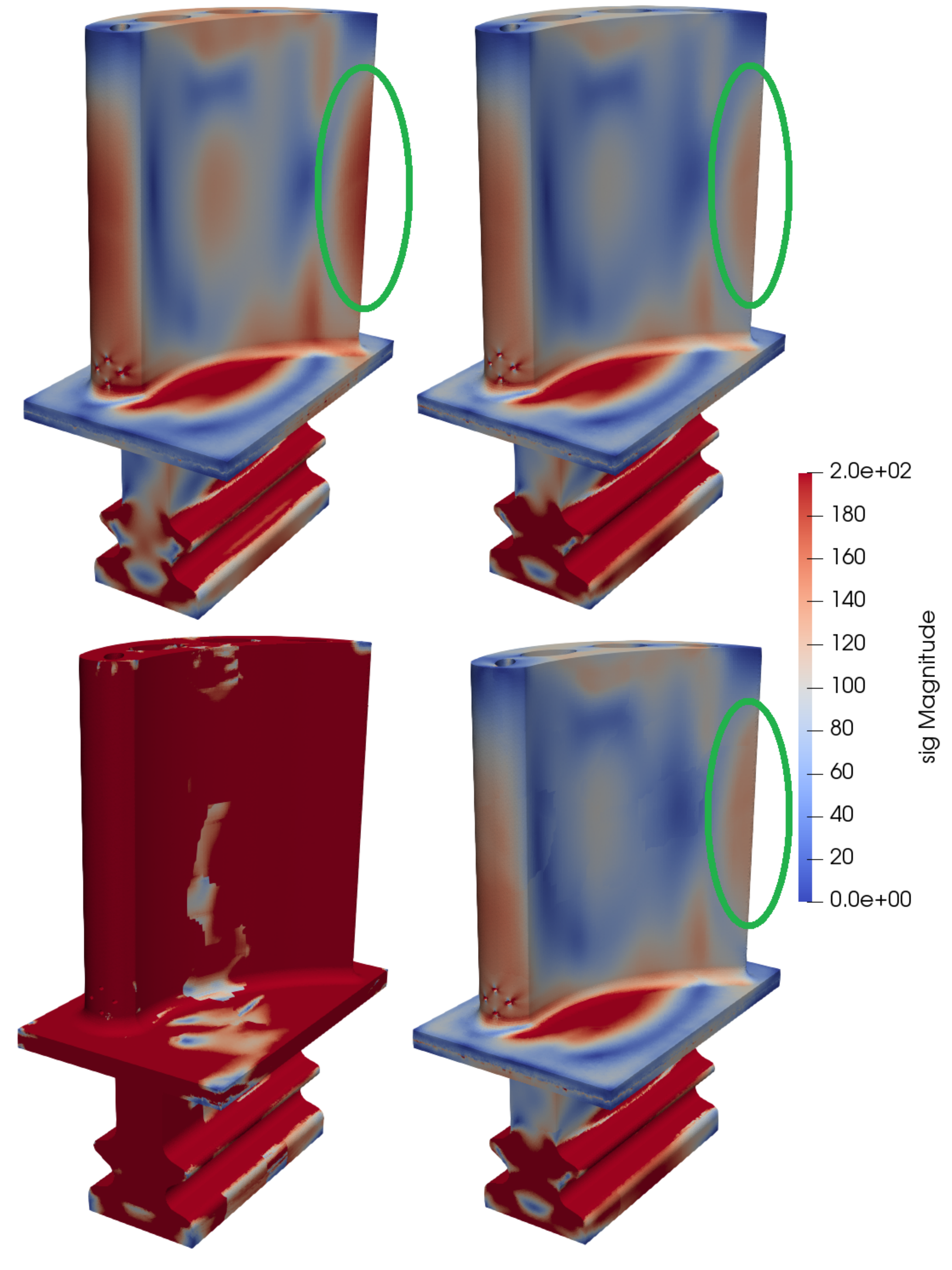
  \caption{Complete ROM dual quantity fields at $t=43.5s$, with enrichment by monitoring subdomain 28: left cumulated plasticity, right magnitude of the stress tensor.}
  \label{fig:bladeRes3}
\end{figure}

{The test cases presented in this section enable to make the two following observations:
\begin{itemize}
\item[\textbf{[O1]}] in the \textit{a posteriori} reduction of elastoviscoplastic computation, \textit{online} variabilities of the temperature loading not encountered during the \textit{offline} stage can lead to important errors,
\item[\textbf{[O2]}] the ROM-Gappy-POD residual~\eqref{eq:error_ind_1} is highly correlated to the error~\eqref{eq:error_1}, so that the proposed error indicator~\eqref{eq:indicator} can be used in the \textit{online} stage as described in the workflow illustrated in Figure~\ref{fig:workflow_online} to correct \textit{online} variabilities of the temperature loading not encountered during the \textit{offline} stage.
\end{itemize}
}

\section{Conclusion and outlook}
\label{sec:conclusion}

In this work, we considered the model order reduction of structural mechanics with elastoviscoplastic behavior laws, with dual quantities such as cumulated plasticity and stress tensor as quantities of interest. We observed in our numerical experiments a strong correlation between the ROM-Gappy-POD residual of the reconstruction of these dual quantities and the global error. From this observation, we proposed an efficient error indicator by means of Gaussian process regression from the data acquired when solving the high-fidelity problem in the learning phase of the reduced order modeling. We illustrated the ability of the error indicator to enrich a reduced order model when the \textit{online} variability cannot be predicted using the current reduced order basis, leading to an accurate reduced prediction.

For the moment, the reduced order model is enriched using a complete reference high-fidelity computation, and the POD and Gappy-POD are recomputed. In a future work, we need to consider restart strategies to call the high-fidelity solver only at the time step of enrichment, from a complete mechanical state reconstructed from the prediction of the reduced order model at the previous time step, which can introduce additional errors. We also need to consider incremental strategies for the POD and Gappy-POD updates.

%\vspace{0.5cm}

\section*{Acknowledgement}
{This research was funded by the French Fonds Unique Interministériel (MOR\_DICUS).}

%\vspace{0.5cm}

\section*{Abbreviations and notations}

{The following abbreviations are used in this manuscript:\\

\noindent 
\begin{tabular}{@{}ll}
POD & Proper Orthogonal Decomposition\\
HF(M) & High-Fidelity (Model)\\
ROM & Reduced Order Model\\
\end{tabular}

\vspace{0.5cm}

\noindent{The following notations are used in this manuscript:}\\

\noindent 
{
\begin{tabular}{@{}ll}
$u$ & high-fidelity displacement field\\
$\hat{u}$ & reduced displacement field\\
$p$ & high-fidelity cumulated plasticity field\\
$\tilde{p}$ & reduced cumulated plasticity field reconstructed by Gappy-POD\\
$\boldsymbol{p}$ & vector of component the value of the high-fidelity cumulated plasticity field at the reduced\\
&integration points\\
$\hat{\boldsymbol{p}}$ & vector of component the cumulated plasticity computed by the behavior law solver at the\\
&reduced integration  points during the online phase. Notice that this vector is not obtained by taking\\
&the value of some field at the reduced integration points.\\
$\tilde{\boldsymbol{p}}$ & vector of component the value of the reduced cumulated plasticity field reconstructed by Gappy-POD\\
&at the reduced integration points\\
$E^p$ & relative error, defined in~\eqref{eq:error_1}\\
$\mathcal{E}^p$ & ROM-Gappy-POD residual, defined in~\eqref{eq:error_ind_1}\\
${\rm Gpr}^p\left(\mathcal{E}^p\right)$ & proposed error indicator, defined in~\eqref{eq:indicator}\\
$p_{\rm off}$ & reference high-fidelity cumulated plasticity field at the considered \textit{offline} variability\\
$p_{\rm ref}$ & reference high-fidelity cumulated plasticity field at the considered \textit{online} variability\\
$\tilde{p}_{\rm nores}$ & reduced cumulated plasticity field reconstructed by Gappy-POD without enrichement (no restart)\\
\end{tabular}
The same notations as the ones on the cumulated plasticity are used for all the dual quantities.
}
}

\bibliographystyle{plain}
\bibliography{biblio}
\end{document}